\documentclass[a4paper, leqno, 11pt]{amsart}
\usepackage[utf8]{inputenc}
\usepackage[a4paper, left=90pt, right=90pt, bottom=100pt, top=100pt]{geometry}

\usepackage{amsthm}
\usepackage{amssymb}
\usepackage[american]{babel}
\usepackage{exscale}
\usepackage[mathscr]{euscript}
\usepackage{url}
\usepackage[all,ps,tips,tpic]{xy}
\usepackage{graphicx}
\usepackage{color}

\usepackage{tikz} 
\usetikzlibrary{matrix,arrows}

\usepackage{amscd}

\numberwithin{equation}{section}

\DeclareMathOperator{\adiag}{anti-diag}
\DeclareMathOperator{\adm}{Adm}
\DeclareMathOperator{\Aut}{Aut}

\DeclareMathOperator{\codim}{codim}
\DeclareMathOperator{\coker}{Coker}
\DeclareMathOperator{\diag}{diag}
\DeclareMathOperator{\GL}{GL}
\DeclareMathOperator{\GSp}{GSp}
\DeclareMathOperator{\id}{Id}
\DeclareMathOperator{\INV}{INV}

\DeclareMathOperator{\LGSp}{LGSp}
\DeclareMathOperator{\pLGSp}{L^{+}GSp}

\DeclareMathOperator{\MT}{MT}
\DeclareMathOperator{\ord}{ord}
\DeclareMathOperator{\sloc}{loc}
\DeclareMathOperator{\dR}{DR}
\DeclareMathOperator{\DR}{H^1_{DR}}
\DeclareMathOperator{\F}{F}

\DeclareMathOperator{\R}{R}

\DeclareMathOperator{\sign}{sign}

\DeclareMathOperator{\Spec}{Spec}

\DeclareMathOperator{\Z}{Z}

\newcommand{\bo}{\mathbf{1}}
\newcommand{\bM}{\mathbf{M}}
\newcommand{\bx}{\mathbf{x}}

\newcommand{\cA}{\mathcal{A}}
\newcommand{\cB}{\mathcal{B}}
\newcommand{\cF}{\mathcal{F}}

\newcommand{\cl}{\ell}
\newcommand{\co}{\mathcal{O}}

\newcommand{\DD}{\mathbb{D}}

\newcommand{\FF}{\mathbb{F}}

\newcommand{\RR}{\mathbb{R}}
\newcommand{\ZZ}{\mathbb{Z}}

\newcommand{\Adm}{\adm (\mu )}
\newcommand{\Fl}{\mathcal{F} \hspace{-0,7mm} \ell}
\newcommand{\pot}[1]{ [\hspace{-0,5mm}[ {#1} ]\hspace{-0,5mm}] }
\newcommand{\rpot}[1]{ (\hspace{-0,7mm}( {#1} )\hspace{-0,7mm}) }

\newcommand{\loc}{\bM_{I,\FF}^{\sloc}}
\newcommand{\Lra}{\Leftrightarrow}
\newcommand{\teil}[1]{(\hspace{-0,3mm}{#1}\hspace{-0,3mm})}
\newcommand{\mycomment}[1]{}

\newenvironment{myeqnarray}{
 
 \begin{displaymath}
 \begin{array}{rcll}}{
 \end{array}
 \end{displaymath}}
 
 \newenvironment{equivarray}{
 
 \begin{displaymath}
 \begin{array}{lc}}{
 \end{array}
 \end{displaymath}}

\newenvironment{twoeqnarray}{
 
 \begin{displaymath}
 \begin{array}{rclrcl}}{
 \end{array}
 \end{displaymath}}

\newtheorem{theorem}{Theorem}[section]
\newtheorem{lemma}[theorem]{Lemma}
\newtheorem{corollary}[theorem]{Corollary}
\newtheorem{proposition}[theorem]{Proposition}

\theoremstyle{definition}
\newtheorem{definition}[theorem]{Definition}
\theoremstyle{remark}
\newtheorem*{remark}{Remark}

\title[The $p$-rank stratification on $\cA_I$]{The $p$-rank stratification on the Siegel moduli space with Iwahori level structure}
\author[P. Hamacher]{by Paul Hamacher}
\begin{document}

 \address{Paul Hamacher \\ Mathematisches Institut der Universit\"at Bonn\\ Endenicher Allee 60 \\ 53115 Bonn \\ Germany}
 \email{hamacher@math.uni-bonn.de}

 \setcounter{section}{-1}

 \begin{abstract}
  Our concern in this paper is to describe the $p$-rank statification on the Siegel moduli space with Iwahori level structure over fields of positive characteristic. We calculate the dimension of the strata and describe the closure of a given stratum in terms of $p$-rank strata. We also examine the relationship between the $p$-rank stratification and the Kottwitz-Rapoport stratification.
 \end{abstract}
 
 \maketitle
 
 \section{Introduction}

 Fix a prime $p$, a positive integer $g$ and an algebraic closure $\FF$ of $\FF_p$. We denote by $\cA_g$ the moduli space of principally polarized abelian varieties of dimension $g$ over $\FF$ and by $\cA_I$ the moduli space of principally polarized abelian varieties of dimension $g$ over $\FF$ with Iwahori level structure at $p$. Let $\pi: \cA_I \longrightarrow \cA_g$ be the canonical projection (see subsection \ref{modsp} for details).
 
 We denote by $\cA_g^{(d)}$ the subset of $\cA_g$ of points that correspond to an abelian variety of $p$-rank $d$. Koblitz showed in his paper \cite{Ko} that these sets form a stratification on $\cA_g$, more precisely:
 
 \begin{theorem} \label{koblitz}
  Let $d \leq g$ be a nonnegative integer. \\ \smallskip
 \emph{(1)} $\cA_g^{(d)}$ is locally closed \\ \smallskip
 \emph{(2)} $\overline{\cA_g^{(d)}} = \bigcup_{d' \leq d} \cA_g^{(d')}$ \\ \smallskip
 \emph{(3)} $\cA_g^{(d)}$ is equidimensional of codimension $g-d$
 \end{theorem}
 
 Now consider the preimages $\cA_I^{(d)} := \pi^{-1}(\cA_g^{(d)})$. These sets form a decomposition of $\cA_I$ into locally closed subsets but the closure of $\cA_I^{(d)}$ can in general not be written as union of sets $\cA_I^{(d')}$. We will nevertheless call them $p$-rank strata and refer to this decomposition as the $p$-rank stratification of $\cA_I$. In this paper we deduce statements on the $p$-rank stratification on $\cA_I$ similar to those of Theorem \ref{koblitz}.
 
 The dimension of the stratum corresponding to $p$-rank zero was already calculated by G\"ortz and Yu in section 8 of \cite{GY2}. Our proof of the dimension formula for arbitrary $p$-rank (see below) is a generalization of theirs. 
 
 Our approach  is via the study of the Kottwitz-Rapoport stratification (KR stratification) on $\cA_I$, which is given by the relative position of the chain of de Rham cohomology groups and the chain of Hodge filtrations associated to a point of $\cA_I$. We use the result of Ng\^o and Genestier which states that the $p$-rank on a KR stratum is constant and thus the KR stratification is a refinement of the $p$-rank stratification. The KR strata are in canonical one-to-one correspondence with a subset $\Adm$ of the extended affine Weyl group of $\GSp_{2g}$ (defined in section \ref{GSp}). Most importantly the dimension, relative position and the $p$-rank on KR strata can be expressed in combinatorial or numerical terms on $\Adm$ (for details, see section \ref{sectKR}).
 
 Denote the integer part of a number $x$ by $\lfloor x \rfloor$. The following theorem summarizes the main results of this paper.
 
 \begin{theorem} \label{results}
 Let $d\leq g$ be a nonnegative integer. Denote by $M^{(d)}$ the subset of elements $x \in \Adm$ which correspond to a KR stratum $\cA_x$ which is top-dimensional inside $\cA_I^{(d)}$. \\ \smallskip
 \emph{(1)} $ \codim \cA_I^{(d)} = \lfloor \frac{g-d}{2} + \frac{1}{2} \rfloor$ \\ \smallskip
 \emph{(2a)} If $g-d$ is even,
  \[
   \overline{\cA_I^{(d)}} = \bigcup_{d'\leq d} \cA_I^{(d')}.
  \]
 \emph{(2b)} If $g-d \not= 1$ is odd,
  \[
   \overline{\cA_I^{(d)}} = \bigcup_{d'\leq d} \cA_I^{(d')} \setminus \bigcup_{x \in M^{(d'')} \atop d''<d,\, 2 | g-d''} \cA_x.
  \]  
  \emph{(2c)} Using the standard embedding $W\subset S_{2g}$ (cf.\ section \ref{GSp}), we have 
  \[
   \overline{\cA_I^{(g-1)}} = \bigcup_{x=t^{x_0}w\in \Adm \atop w(\{1,\ldots,g\}) = \{1,\ldots,g\}} \cA_x.
  \]
 \end{theorem}
 
 In contrast to the the $p$-rank strata on $\cA_g$, the $p$-rank strata on $\cA_I$ are in general not equidimensional. 
 
 This paper is subdivided into three parts. In the first part we give some background. Here we give the definition of the extended affine Weyl group and a characterization of $\Adm$ in section \ref{GSp} and give the construction and required properties of the KR stratification in \ref{sectKR}. The second part is the calculation of the dimension of the $p$-rank strata. Finally, we compare the KR stratification and the $p$-rank stratification in the third part, giving an explicit description of $M^{(d)}$ in section \ref{secttopdim} and proving part (2) of Theorem \ref{results} in \ref{sectrelpos}.
 
 \textbf{Acknowledgements:} I would like to express my sincere gratitude to my advisors M.~Rapoport and E.~Viehmann for their steady encouragement and interest in my work. I am thankful that they introduced me to this subject and helped me in understanding it. Moreover I warmly thank U.~G\"ortz for his helpful remarks and for supplying me with his calculation of $\mu$-admissible elements in the case $g=3$ which ultimately led me to the description of the closure of $\cA_I^{(g-1)}$ as above. I am grateful to P.~Hartwig for his explanations about affine flag varieties. 
 
 \section{Background}
 We keep the notation of the introduction.  We also fix the following notation. For any set $R$ and $n$-tuple $v \in R^n$ we denote the $i$-th component of $v$ by $v(i)$. We abbreviate a tuple of the form $(v_1, \ldots , v_1, v_2, \ldots , v_2 , \ldots , v_m, \ldots, v_m)$ by $(v_1^{(k_1)}\cdots v_m^{(k_m)})$ where $k_i$ denotes the multiplicity of $v_i$. If $k_i = 1$ we will omit it.
  If $R$ is an ordered set, e.g. $R=\ZZ$, and $u,v \in R^n$ we write $u \leq v$ iff $u(i) \leq v(i)$ for every $i = 1,\ldots,n$.
 
  In order to simplify equations when using case analysis, we define for any statement $P$ the term
 \[
  \delta_P := \left\{ \begin{array}{rl}
                      0 & \textnormal{ if } P \textnormal{ is false} \\
                      1 & \textnormal{ if } P \textnormal{ is true}.
                     \end{array} \right.
 \]

 \subsection{Preliminaries on $\GSp_{2g}$} \label{GSp}
 
 \subsubsection{Group theoretic notation} \label{GSp1}
 
 Let $G$ be a reductive linear algebraic group over an algebraically closed field $k$, and let $A$ be a maximal torus of $G$. These data give rise to a root datum $(X^*(A),R_G,X_*(A),R_G^\vee)$ and its Weyl group $W_G = N_GA/A$. We denote by $Q_G^\vee$ its coroot lattice and by $W_{a\, G} = Q_G^\vee \rtimes W_G$ the affine Weyl group of $G$. For an element $x_0 \in Q_G^\vee$ we denote by $t^{x_0}$ the corresponding element in $W_{a\, G}$. The choice of a Borel subgroup containing $A$ determines a set of positive roots $R^+_G$ and a set of simple roots $\Delta_G$. Associated to $\Delta_G$ we have the sets of simple reflections $S_G = \{s_\alpha;\, \alpha \in \Delta_G\}$ and $S_{a\, G} = S_G \cup \{t^{-\widetilde{\alpha}^\vee}\cdot s_{\widetilde{\alpha}}\}$ of $W_G$ and $W_{a\, G}$ where $\widetilde{\alpha}$ denotes the (unique) highest root of $R^+_G$. Applying the standard identification of $W_{a\, G}$ with the set of alcoves on $X_* (A)_\RR$, our choice of $S_{a\, G}$ places the base alcove in the anti-dominant chamber. We denote by $\cl$ and by $\leq$ the length function and the Bruhat order on the Coxeter system $(W_{a\, G}, S_{a\, G})$. Whenever we deal with the case $G=\GSp_{2g}$ (as it will be in the majority of cases), we drop the subscript $G$.

 Let $\GL_{2g}$ denote the general linear group over $\FF$ and $D$ be its diagonal torus. We use the standard identification of the cocharacter group $X_*(D)$ with $\ZZ^{2g}$ and of $W_{\GL_{2g}}$ with the symmetric group $S_{2g}$.
 
 Denote by $\GSp_{2g}$ the group of similitudes corresponding to a $2g$-dimensional symplectic $\FF$-vector space $(V,\psi)$. We embed the $\GSp_{2g}$ into $\GL_{2g}$ by choosing a Darboux basis of $V$, i.e. a basis $(e_1, \ldots , e_{2g})$ such that
 \[
  \psi(e_i,e_{2g+1-i}) = - \psi (e_{2g+1-i}, e_i) = 1
 \]
 for $1\leq i \leq g$ and $\psi (e_i,e_j) = 0$ otherwise. Then the subgroup $T \subset \GSp_{2g}$ of diagonal matrices is a maximal torus and the upper triangular matrices form a Borel subgroup $B$ of $\GSp_{2g}$.
 
 $T$ is the group of all elements $t$ of the form $\diag(t_1, \ldots , t_{2g})$ such that there is a $c(t) \in k^\times$ with $t_i \cdot t_{2g+1-i} = c(t)$ for all $i = 1, \ldots , g$. Hence the embedding of $X_*(T)$ into $X_*(D) = \ZZ^{2g}$ yields the identification
 \[
  X_* (T) = \{v\in\ZZ^{2g};\, v(1)+v(2g) = v(2)+v(2g-1) = \cdots = v(g)+v(g+1)\}.
 \]
 We denote by $e_i^* \in X^* (T)$ the character which maps an element $t\in T$ to its $i$-th diagonal entry and $c \in X^* (T)$ the character which maps $t$ to its similitude factor. The positive roots in this setup are
 \begin{myeqnarray}
  \beta_{i,j}^1 &=& e_i^* - e_j^* & 1\leq i < j \leq g \\
  \beta_{i,j}^2 &=& e_i^* + e_j^* - c \quad & 1\leq i < j  \leq g \\
  \beta_i^3 &=& 2e_i^* - c & 1 \leq i \leq g.
 \end{myeqnarray}
 Then the simple roots are $\beta_{i, i+1}^1$ with $1\leq i \leq g-1$ and $\beta_g^3$.  
 
 The embedding of $\GSp_{2g}$ into $\GL_{2g}$ induces an embedding of the Weyl group $W$ of $\GSp_{2g}$ into $W_{\GL_{2g}} = S_{2g}$. Then $W$ is the centralizer of
 \[
  \theta = (1\enspace 2g)(2\enspace 2g-1) \cdots (g\enspace g+1)
 \]
 or equivalently the subset of all elements $w$ which satisfy $w(2g+1-i) = 2g+1-w(i)$ for all $i$. The simple (affine) reflections in this setup are
 \begin{myeqnarray}
  s_0 &=& t^{(-1 \, 0^{(2g-2)} \, 1)}\cdot (1\enspace 2g) & \\
  s_{i,i+1}  &=& (i\enspace i+1)(2g-i\enspace 2g-i+1)  \quad & 1 \leq i \leq g-1 \\
  s_g &=& (g\enspace g+1) &
 \end{myeqnarray}
 
 \subsubsection{The extended affine Weyl group}
 Using the same notation as above, we call $\widetilde{W}_G := X_* (A) \rtimes W_G \cong N_G A(k\rpot{t})/A(k\pot{t})$ the extended affine Weyl group of $G$. Analogous to the case of the affine Weylgroup, we denote by $t^{x_0}$ the element of $\widetilde{W}$ which corresponds to the cocharacter $x_0 \in X_* (T)$. Since every $W_G$-orbit of $X_* (T)$ is contained in a $Q_G^\vee$ coset, the affine Weyl group is a normal subgroup of $\widetilde{W}_G$ and we get a short exact sequence
 \[
  0 \longrightarrow W_{a\, G}\longrightarrow \widetilde{W}_G \longrightarrow X_*(A)/Q_G^\vee\longrightarrow 0.
 \]
 Identifying $t^{x_0} \cdot w$ with the map $x \mapsto w\cdot x + x_0$, we consider $\widetilde{W}_G$ as subgroup of the group of affine transformations on ${X_*(A)_\RR}$. Then the action of $\widetilde{W}_G$ on $X_* (A)_\RR$ stabilizes the union of all affine hyperplanes corresponding to an affine root. Thus we get a transitive action of $\widetilde{W}_G$ on the set of alcoves. So the short exact sequence right-splits; $\widetilde{W}_G$ is the semidirect product $W_{a\, G} \rtimes \Omega_G$ where $\Omega_G$ is the subgroup of all elements which fix the base alcove.
 
 We define the length function and the Bruhat order on $\widetilde{W}_G$ as follows: Let $x=c_1w_1$ and $y_1=c_2w_2$ be two elements of $\widetilde{W}_G$ and $c_i$ resp. $w_i$ their $\Omega_G$- and $W_{a\, G}$-component. We say that $x \leq y$  if $c_1 = c_2$ and $w_1 \leq w_2$ w.r.t. the Bruhat order on $W_{a\, G}$. The length $\cl(x)$ is defined to be $\cl(w_1)$. 
 
 Since we have $W \subset W_{\GL_{2g}}$ and $X_*(T) \subset X_*(D)$, the extended affine Weyl group $\widetilde{W}$ of $\GSp_{2g}$ is a subgroup of $\widetilde{W}_{\GL_{2g}}$.
 
 Now we recall the description of the extended affine Weyl group in terms of extended alcoves as in \cite{GY1}. For this purpose denote by $(e_i)_{i=1,\ldots,2g}$ the family of standard basis vectors in $\ZZ^{2g}$ and let $\bo = (1, 1, \ldots ,1) \in\ZZ^{2g}$.
 
 \begin{definition}
  An extended alcove is a tuple of vectors $\bx = (x_0, \ldots ,x_{2g-1})$ in $X_* (D)$ such that for every $i$ there is a $w(i) \in \{1,\ldots 2g\}$ with
  \[
   x_i = x_{i-1}-e_{w(i)}
  \]
  where $x_{2g} := x_0 - \bo$.
 \end{definition}
 
 Then $\widetilde{W}_{\GL_{2g}}$ acts simply transitively on the set of extended alcoves. Using $\omega = (\omega_0, \omega_1, \omega_2, \ldots) = (0,-e_1,-e_1-e_2,\ldots )$ as base point, we identify the extended affine Weyl group with this set. Using the same notation as in the definition, an extended alcove $\bx$ corresponds to $t^{x_0}\cdot w \in \widetilde{W}_{\GL_{2g}}$. We call an extended alcove corresponding to an element of $\widetilde{W}$ a $G$-alcove. Note that an extended alcove $\bx = (x_0,\ldots ,x_{2g-1})$ is a $G$-alcove if and only if there is a $c\in\ZZ$ with
 \[
  x_i + \theta(x_{2g-i}) = c\cdot\bo
 \]
 for all $i$.
 
 Observe that since $\omega$ lies in the closure of the base alcove, every extended alcove lies in the closure of the corresponding alcove of the $W_{a \, \GL_{2g}}$-component. This observation and the fact that the length function and the Bruhat order of the (affine, extended affine) Weyl group of $\GSp_{2g}$ is inherited from the (affine, extended affine) Weyl group of $\GL_{2g}$ (see \cite{KR}, \S 4.1) imply that we can conclude the following lemmata from their analogue concerning the affine Weyl group.
 
 \begin{lemma}\label{wall}
  Let $\alpha$ be an affine root of $\GSp_{2g}$ with corresponding wall $H_\alpha$ and reflection $s_\alpha$. If $x$ and $y = s_\alpha\cdot x$ are two elements of $\widetilde{W}$, we have $x \leq y$ if and only if $\bx$ lies on the same side of $H_\alpha$ as the base alcove.
 \end{lemma}
 
 \begin{lemma}[Iwahori-Matsumoto formula] \label{IMF}
  The length of an element of $\widetilde{W}$ equals the number of walls that separate the corresponding $G$-alcove from the base alcove, i.e.
  \[
   \cl(t^{x_0}\cdot w) = \sum_{\beta\in R^+ \atop w^{-1}\beta \in R^+} |\langle\beta,x_0\rangle| + \sum_{\beta\in R^+ \atop w^{-1}\beta \not\in R^+} |\langle\beta,x_0\rangle+1|.
  \]
 \end{lemma}
  
 The analogue of Lemma \ref{IMF} is the usual Iwahori-Matsumoto formula for the affine Weyl group, the analogue of Lemma \ref{wall} is stated and proven in \cite{KR}, Corollary 1.5.
 
 \subsubsection{Minuscule $G$-alcoves} \label{ss minuscule}
 As we will see in the next section, the KR strata are in one-to-one correspondence with the minuscule $G$-alcoves of size $g$. In this subsection we give two useful characterizations of an element of $\widetilde{W}$ which corresponds to a minuscule alcove of size $g$.
 
 \begin{definition}
  \emph{(1)} We call a $G$-alcove $\bx$ minuscule of size $g$ if
  \begin{equation} \label{minuscule}
   \omega_i \leq x_i \leq \omega_{i} + \bo \textnormal{ for all } i 
  \end{equation}   \begin{equation}\label{size}
   \{x_0(i),x_0(2g+1-i)\} = \{0,1\} \textnormal{ for all } i
  \end{equation}
  \emph{(2)} An element $x$ of $\widetilde{W}$ is called $\mu$-admissible for $\mu \in X_* (T)$ if there exists a $w\in W$ such that $x \leq w(\mu)$. We denote the set of $\mu$-admissible elements by $\Adm$. 
 \end{definition} 
 
 \begin{remark}
  (\ref{size}) excludes merely the alcoves $\omega$ and $(\omega_0 + \bo, \ldots, \omega_{2g-1} + \bo)$.
 \end{remark}
 
 \begin{remark}
  It is obvious from the definition that there exists a (unique) element $\tau\in\Omega$ such that $\Adm \subset W_a\tau$. It is given by
 \[
  \tau = t^{(0^{(g)}\, 1^{(g)})}\cdot \left( (1\enspace g+1)\cdots (g\enspace 2g) \right) .
 \]
 \end{remark}
 
 \begin{lemma} \label{mu-adm}
  If (\ref{size}) holds, we can replace (\ref{minuscule}) by
  \begin{equation} \label{newminuscule}
   x_{0}(i) = \left\{ \begin{array}{ll} 0 & w^{-1}(i) > i \\ 1 & w^{-1}(i) < i \end{array} \right.
  \end{equation}
  where $t^{x_0}\cdot w$ corresponds to $\bx$.
 \end{lemma}

 \begin{proof}
  Assume (\ref{minuscule}) holds. Then
  \[
   w^{-1}(i) > i \Rightarrow x_0(i) = x_{w^{-1}(i)-1} (i) \leq \omega_{w^{-1}(i)-1} (i)+1 = 0
  \]
  \[
   w^{-1}(i) < i \Rightarrow x_0(i) = x_{w^{-1} (i)}+1 \geq \omega_{w^{-1}(i)} (i)+ 1 = 1
  \]
  On the other hand if (\ref{newminuscule}) holds, we get
  \[
   x_k(i) - \omega_k (i) = x_0 (i) - \delta_{w^{-1}(i) \leq k} + \delta_{i \leq k} = \left\{ \begin{array}{rcl}
                                                                                            1 &\textnormal{if } w^{-1}(i) > i, & i \leq k < w^{-1}(i) \\
                                                                                            0 &\textnormal{if } w^{-1}(i) > i, &  \textnormal{otherwise} \\
                                                                                            0 &\textnormal{if } w^{-1}(i) < i, & w^{-1} (i) \leq k < i \\
                                                                                            1 &\textnormal{if } w^{-1}(i) < i, &  \textnormal{otherwise} \\
                                                                                            x_0(i) &\textnormal{if } w^{-1}(i) = i. & \\
                                                                                           \end{array} \right. 
  \]                                                                                      
 \end{proof}

 A more vivid description is given by a result of Kottwitz and Rapoport. For this let us denote $\mu = (1^{(g)}\, 0^{(g)}) \in X_* (T)$.
 \begin{theorem}[\cite{KR}, Theorem 4.5.3] \label{KR453}
  A $G$-alcove $\bx$ is minuscule of size $g$ if and only if it corresponds to a $\mu$-admissible element.
 \end{theorem}
 
 \subsection{The KR stratification on $\cA_I$} \label{sectKR}
  In this section we recall the construction of the KR stratification and some geometric properties of KR strata resp.\ chains of abelian varieties corresponding to a point of a given KR stratum.
  
 \subsubsection{Moduli spaces} \label{modsp}

 Let $N \geq 3$ be an integer coprime to $p$. The ``classical'' Siegel moduli problem associates the set of isomorphism classes $(A,\lambda,\eta)$ to a locally Noetherian $\FF$-scheme $S$ where
 \begin{itemize}
  \item $A$ is an abelian variety of relative dimension $g$ over $S$,
  \item $\lambda:A\longrightarrow A^\vee$ is a principal polarization,
  \item $\eta$ is a symplectic level-$N$-structure on $A$.
 \end{itemize}
 The moduli problem is solved by an irreducible quasi-projective $\FF$-scheme $\cA_g$ of dimension $\frac{g(g+1)}{2}$.
 
 Now consider the functor associating the set of isomorphism classes of quadruples $(A_\bullet,\lambda_0,\lambda_g,\eta)$ to a locally Noetherian $\FF$-scheme $S$ where
 \begin{itemize}
  \item $A_\bullet = (A_0 \stackrel{\alpha}{\longrightarrow} A_1 \stackrel{\alpha}{\longrightarrow} \cdots \stackrel{\alpha}{\longrightarrow} A_g)$ is a sequence of abelian varieties over $S$ of relative dimension $g$ and the $\alpha$ are isogenies of degree $p$,
  \item $\lambda_0,\lambda_g$ are principal polarizations of $A_0$ and $A_g$ respectively,
  \item $\eta$ is a symplectic level-$N$-structure on $A_0$,
 \end{itemize}
 such that the composition of all arrows in the diagram
  \begin{equation} 
  \begin{tikzpicture} \label{diagIwahori}
   \matrix(a)[matrix of math nodes, row sep=3.0em, column sep=2.5em,text height=1.5ex, text depth=0.45ex]
   {A_0 & A_1 & \cdots & A_g \\ A_0^\vee & A_1^\vee & \cdots & A_g^\vee\\};
   \path[->] (a-1-1) edge node[above] {$\alpha$} (a-1-2);
   \path[->] (a-1-2) edge node[above] {$\alpha$} (a-1-3);
   \path[->] (a-1-3) edge node[above] {$\alpha$} (a-1-4);
   \path[->] (a-2-2) edge node[above] {$\alpha^\vee$} (a-2-1);
   \path[->] (a-2-3) edge node[above] {$\alpha^\vee$} (a-2-2);
   \path[->] (a-2-4) edge node[above] {$\alpha^\vee$} (a-2-3);
   
   \path[->] (a-2-1) edge node[left] {$\lambda_0^\vee$} (a-1-1);
   \path[->] (a-1-4) edge node[right] {$\lambda_g$} (a-2-4);
  \end{tikzpicture}
 \end{equation}
 equals multiplication by $p$. This functor is represented by a quasi-projective $\FF$-scheme $\cA_I$ of dimension $\frac{g(g+1)}{2}$ and the canonical projection $\pi: \cA_I \rightarrow \cA_g, (A_\bullet,\lambda_0,\lambda_g,\eta)\mapsto (A_0,\lambda_0,\eta)$ is a proper and surjective morphism.
 
 Let us abbreviate the notation of an $S$-point of $\cA_I$. We will usually denote it by $\underline{A_\bullet}$. In the following we write $A_\bullet = A_0 \rightarrow \ldots \rightarrow A_{2g}$ for the sequence of abelian varieties of $\underline{A_\bullet}$ \emph{supplemented by its dual}. This is, we identify $A_g \cong A_g^\vee$ via $\lambda_g$ and let $A_{2g-i} := A_{i}^\vee$ for $i=0,\ldots ,g$. The morphisms $A_{2g-i} \rightarrow A_{2g-i+1}$ are defined to be the dual of $A_{i-1} \rightarrow A_i$.
 
  Given an abelian variety $A$, we denote by $A[p]$ the kernel of the multiplication by $p$. Recall that if $A$ is defined over an algebraically closed field $k$ of characteristic $p$, the group of $A[p](k)$ is isomorphic to $(\ZZ/p\ZZ)^r$ where $0 \leq r \leq \dim A$ is called the $p$-rank of $A$.
 
 Since the preimage of a locally closed subset w.r.t.\ a continuous map is again locally closed, the $p$-rank strata on $\cA_I$ are locally closed. We endow them with the reduced subscheme structure. 

 \subsubsection{Standard lattice chains and the first de Rham cohomology group} \label{ssectlc}
 We begin with the definition of the standard lattice chains over $\FF\pot{t}$ and $\FF$. Let $e_1,\ldots,e_{2g}$ denote the canonical basis of $\FF\rpot{t}^{2g}$. The standard lattice chain over $\FF\pot{t}$
 \[
  \lambda_\bullet = \lambda_{-2g}\ \hookrightarrow \lambda_{-2g+1} \hookrightarrow \ldots \hookrightarrow \lambda_{0}
 \]
 is given by $\lambda_{-i} = \langle e_1, \ldots, e_{2g-i}, t\cdot e_{2g-i+1},\ldots ,t\cdot e_{2g}\rangle_{\FF\pot{t}} \subset \FF\rpot{t}^{2g}$ where $\{e_1,\ldots ,e_{2g}\}$ is the standard basis of $\FF\rpot{t}^{2g}$ and $\langle \ldots \rangle_{\FF\pot{t}}$ denotes the $\FF\pot{t}$-submodule which is generated by the elements inside the brackets. Denote by $(\,\, ,\, )$ the bilinear form represented by the matrix $J_{2g} := \adiag (-1^{(g)}\, 1^{(g)})$. Obviously $\lambda_{-g-i}$ is the dual of $\lambda_{-g+i}$ w.r.t.\ $t^{-1}\cdot(\,\,,\,)$. We denote by
 \[
  \Lambda_\bullet = \Lambda_{-2g} \rightarrow \Lambda_{-2g+1} \rightarrow \ldots \rightarrow \Lambda_0
 \]
 the lattice chain obtained by base changing $\lambda_\bullet$ to $\FF$ with respect to the evaluation at zero $ev_0: \FF\pot{t} \rightarrow \FF$. The images of the canonical bases of $\lambda_{-i}$ w.r.t.\ the base change morphism are bases $\{e_1^i,\ldots, e_{2g}^i\}$ of $\Lambda_{-i}$ such that the linear map $\Lambda_{-i} \rightarrow \Lambda_{-i+1}$ is represented by the matrix $\diag (1^{(2g-i)}\, 0 \, 1^{(i-1)})$. We identify $\Lambda_{i-2g} \cong \Lambda_{-i}^\vee$ for all $i=0,\ldots, g$ via the non-degenerate bilinear form $\Lambda_{i-2g} \times \Lambda_{-i} \rightarrow \FF$ given by the matrix $J_{2g}$. In terms of this identification the morphism $\Lambda_{i-1-2g} \rightarrow \Lambda_{i-2g}$ is the dual of $\Lambda_{-i} \rightarrow \Lambda_{-i+1}$. Thus we can also write 
 \begin{equation} \label{sequence}
  \Lambda_\bullet = \Lambda_0^\vee \rightarrow \cdots \rightarrow \Lambda_{-g}^\vee \stackrel{J_{2g}}{\cong} \Lambda_{-g} \rightarrow \cdots \rightarrow \Lambda_0 \stackrel{J_{2g}}{\cong} \Lambda_0^\vee
 \end{equation}
 For any $\FF$-algebra $R$ we denote the base change of $\Lambda_\bullet$ and $\lambda_\bullet$ to $R$ resp.\ $R\pot{t}$ by $\Lambda_{\bullet,R}$ resp.\ $\lambda_{\bullet,R}$.
 
 We denote by $H_{\dR}^i (A/S)$ or simply $H_{\dR}^i(A)$ the de Rham cohomology of an $S$-scheme $A \stackrel{a}{\rightarrow} S$. We are interested in the case where $i=1$ and $A$ is an abelian variety of relative dimension $g$. In this case $\DR(A/S)$ is a locally free $\co_S$-module of rank $2g$ and the Hodge - de Rham spectral sequence degenerates at $E_1$, yielding an inclusion $\omega_A := \R^0a_*(\Omega_{A/S}^1) \hookrightarrow \DR(A/S)$. This embedding makes $\omega_A$ Zariski-locally a direct summand of rank $g$ of $\DR (A)$ (\cite{BBM}, \S 2.5). Furthermore, we have a natural isomorphism $\DR(A^\vee/S) \cong \DR(A/S)^\vee$ (\cite{BBM} \S 5.1).
 
 If $S = \Spec k$ is the spectrum of a perfect field of characteristic $p$, the Hodge-filtration of $\DR(A/S)$ can also be given in terms of Dieudonn\'e theory. For this we denote the (contravariant) Dieudonn\'e module of a finite, commutative, $p$-torsion group scheme $K$ over $k$ by $(\DD(K),F,V)$.
  
 \begin{theorem}[\cite{Oda}, Cor.\ 5.11] \label{oda}
  Let $k$ be a perfect field and $A$ an abelian variety over $k$. There is a natural isomorphism $\DR(A/\Spec k ) \cong \DD(A[p])$ which identifies $\omega_A$ with $V\DD(A[p])$.
 \end{theorem}

 Now let $S= \Spec R$ be affine and Noetherian. It is shown in de Jong's paper \cite{dJ} that if we apply $\DR$ to the diagram ($\ref{diagIwahori}$) we get a diagram of $\co_S$-modules
 \[
  \begin{tikzpicture}
   \matrix(a)[matrix of math nodes, row sep=3.0em, column sep=2.5em,text height=1.5ex, text depth=0.45ex]
   {\DR(A_0) & \DR(A_1) & \cdots & \DR(A_{g}) \\ \DR(A_0)^\vee & \DR(A_1)^\vee & \cdots & \DR(A_{g})^\vee \\};
   \path[->] (a-1-2) edge (a-1-1);
   \path[->] (a-1-3) edge (a-1-2);
   \path[->] (a-1-4) edge (a-1-3);
   \path[->] (a-2-1) edge (a-2-2);
   \path[->] (a-2-2) edge (a-2-3);
   \path[->] (a-2-3) edge (a-2-4);  
   \path[->] (a-1-1) edge node[left] {$\cong$} node[right] {$q_0^\vee$} (a-2-1);
   \path[->] (a-2-4) edge node[left] {$\cong$} node[right] {$q_g$} (a-1-4);
  \end{tikzpicture}.
 \]
 such that
 \begin{itemize}
  \item The horizontal sequences are dual to each other.
  \item $q_0$ and $q_g$ define non-degenerate alternating forms on $\DR(A_0)$ and $\DR(A_g)$ respectively.
  \item $\coker (\DR(A_i) \rightarrow \DR(A_{i-1}))$ is a locally free $R$-module of rank $1$ for $i=1,\ldots ,2g$.
  \item The composition of all morphisms in the diagram is zero.
 \end{itemize}
 Any diagram having these properties is locally isomorphic to $\Lambda_{\bullet,R}$. Furthermore $\omega_{A_0}$ and $\omega_{A_g}$ are totally isotropic w.r.t.\ the bilinear forms $q_0$ resp.\ $q_g$ (see \cite{dJ}).
 
 \subsubsection{The local model} \label{ssectlm}
 Following the book of Rapoport and Zink \cite{RZ}, we obtain a local model diagram for $\cA_I$,
 \[
  \begin{tikzpicture}
   \node (A) at (0,0em) {$\widetilde{\cA}_I$};
   \node (B) at (-5em,-5em) {$\cA_I$};
   \node (C) at (5em,-5em) {$\loc$};
   
   \path[->] (A) edge node[above] {$\varphi$} (B);
   \path[->] (A) edge node[above] {$\psi$} (C);
  \end{tikzpicture}
 \]
 i.e.\ a diagram of $\FF$-schemes where $\psi$ is smooth, $\varphi$ is smooth and surjective and $\cA_I \cong \loc$ \'etale locally. 
 
 For any Noetherian $\FF$-algebra $R$, the $R$-valued points of this diagram are given by
 \begin{eqnarray*}
  \widetilde{\cA_I} (R) &=& \{(\underline{A_\bullet}, \iota_\bullet);\, \underline{A_\bullet} \in \cA_I (R), \iota_\bullet: \DR(A_\bullet) \stackrel{\cong}{\rightarrow} \Lambda_{\bullet,R}\} \\
  \loc (R) &=& \left\{ (\cF_{-i})_{i=0,\ldots,2g};\,\begin{array}{l}\forall i: \cF_{-i} \hookrightarrow \Lambda_{-i,R} \textnormal{ is locally a direct summand of rank } g. \\ \forall i<2g: \Lambda_{-i-1} \rightarrow \Lambda_{-i} \textnormal{ maps } \cF_{-i-1} \textnormal{ to } \cF_{-i}. \\ \forall i: \cF_{-i}, \cF_{i-2g} \textnormal{ are in duality w.r.t.\ } \Lambda_{-i}^\vee \cong \Lambda_{i-2g}  \end{array} \right\} \\
  \varphi(\underline{A_\bullet},\iota_\bullet) &=& \underline{A_\bullet} \\
  \psi(\underline{A_\bullet},\iota_\bullet) &=& (\iota_{i} (\omega_{A_i}))_{i=0, \ldots ,2g}.
 \end{eqnarray*}
 
 Here we mean by an isomorphism $\iota_\bullet$ a tuple of $R$-linear isomorphisms $\iota_i: \DR(A_i) \stackrel{\cong}{\rightarrow} \Lambda_{-i}$ which commute with the canonical morphisms $\Lambda_{-i-1} \rightarrow \Lambda_{-i}$ resp. their dual for $0 \leq i < g$ and identify the bilinear forms $q_0$, $q_g$ with those of (\ref{sequence}) up to a constant. 
 Denote by $\underline{\Aut}(\Lambda_\bullet)$ the group scheme over $\FF$ whose $R$-points are the automorphisms of $\Lambda_{\bullet,R}$. Then $\varphi$ is an $\underline{\Aut}(\Lambda_\bullet)$-torsor and $\psi$ is $\underline{\Aut} (\Lambda_\bullet)$-equivariant with respect to the canonical left action on $\widetilde{\cA_I}$ and $\loc$.
 
 We have a bijection between $R$-points of $\loc$ and diagrams of $R\pot{t}$-modules
 \begin{equation} \label{diagFlag}
  \begin{tikzpicture}
   \matrix(a)[matrix of math nodes, row sep=3.0em, column sep=2.5em,text height=1.5ex, text depth=0.45ex]
   {\lambda_{-2g,R} & \lambda_{-2g+1,R} & \cdots & \lambda_{0,R} \\ L_{-2g} & L_{-2g+1} & \cdots & L_0 \\ t\cdot\lambda_{-2g} & t\cdot\lambda_{-2g+1,R} & \cdots & t\cdot\lambda_{0,R} \\};
   \path[right hook->] (a-1-1) edge (a-1-2);
   \path[right hook->] (a-1-2) edge (a-1-3);
   \path[right hook->] (a-1-3) edge (a-1-4);
   \path[right hook->] (a-2-1) edge (a-2-2);
   \path[right hook->] (a-2-2) edge (a-2-3);
   \path[right hook->] (a-2-3) edge (a-2-4);
   \path[right hook->] (a-3-1) edge (a-3-2);
   \path[right hook->] (a-3-2) edge (a-3-3);
   \path[right hook->] (a-3-3) edge (a-3-4);
 
   \path[right hook->] (a-2-1) edge (a-1-1);
   \path[right hook->] (a-3-1) edge (a-2-1);
   \path[right hook->] (a-2-2) edge (a-1-2);
   \path[right hook->] (a-3-2) edge (a-2-2);
   \path[right hook->] (a-2-4) edge (a-1-4);
   \path[right hook->] (a-3-4) edge (a-2-4);
  \end{tikzpicture}
 \end{equation}
 such that
 \begin{itemize}
  \item $\lambda_{-i}/L_{-i}$ are projective $R$-modules of rank $g$.
  \item $L_{-i}$ and $L_{i-2g}$ are dual to each other with respect to the bilinear form $t^{-1}\cdot (\,\, , \,)$.
 \end{itemize}
 Here a sequence $\cF_\bullet \in \loc (R)$ corresponds to a diagram (\ref{diagFlag}) with $L_{-i}$ being the preimage of $\cF_{-i}$ w.r.t.\ the canonical projection $\lambda_{-i,R} \rightarrow \lambda_{-i,R}/t\cdot\lambda_{-i,R} = \Lambda_{-i,R}$. Using that projectivity is a local property, it is easy to see that this indeed defines a bijection of $R$-valued points of $\loc$ and diagrams of this form. Obviously this bijection is functorial.  
  
 \subsubsection{Construction of the KR stratification} \label{ssectflag} Here we recall the construction of the KR stratification of Ng\^{o} and Genestier in \cite{GN}. We denote by $\LGSp_{2g}$ resp.\ $\pLGSp_{2g}$ the loop group resp.\ the positive loop group of $\GSp_{2g}$. Let $\cB$ be the standard Iwahori subgroup, i.e. the preimage of $B$ w.r.t.\ the reduction map $\pLGSp_{2g} \rightarrow \GSp_{2g}$ and $\Fl := \LGSp_{2g}/\cB$ the affine flag variety.
 
 \begin{definition} Let $R$ be a Noetherian $\FF$-algebra. \smallskip \\
  \emph{(1)} A lattice in $R\rpot{t}^{2g}$ is a sub-$R\pot{t}$-module $L$ such that $t^N R\pot{t}^{2g} \subset L \subset t^{-N} R\pot{t}^{2g}$ for some $N$ and such that $t^{-N}R\pot{t}/L$ is a projective $R$-module. \smallskip \\
  \emph{(2)} A complete periodic lattice chain is a sequence of lattices $\{L_i\}_{i\in\ZZ}$ with $L_{i-1} \subset L_{i}$ such that for every $i$ we have that $L_i/L_{i-1}$ is a locally free $R$-module of rank $1$ and $L_{i+2g} = t^{-1}\cdot L_i$. We call $\{ L_i \}_{i\in \ZZ}$ self-dual if Zariski-locally on $R$, there exits a unit $c\in R\rpot{t}$ such that $L_{-i} = c\cdot L_i^\vee$, where $L_i^\vee$ denotes the dual of $L_i$ w.r.t.\ the bilinear form $(\,\, , \,)$.
 \end{definition}
 
 We expand $\lambda_\bullet$ to a self-dual periodic lattice chain by setting $\lambda_{i+2g\cdot r} := t^r \cdot \lambda_i$. It is well-known that the map
 \begin{eqnarray*}
  \Fl(R) &\rightarrow& \{ \textnormal{self-dual complete periodic lattice chains in } R\rpot{t}\} \\
  g\cdot\cB &\mapsto& g\cdot\lambda_\bullet
 \end{eqnarray*}
 is bijective i.e.\ the functor associating the set of self-dual complete periodic lattice chains in $R\rpot{t}^{2g}$ with $R$ is represented by $\Fl$. By expanding the diagram (\ref{diagFlag}) we obtain an embedding $\loc \hookrightarrow \Fl$. We denote by $\Fl_x := \cB x \cB/\cB$ the Schubert cell associated to an element $x \in \widetilde{W}$. Then $\loc$ can be written as disjoint union of Schubert cells.
   
 \begin{proposition}[\cite{GN} Cor.~3.2] \label{cells}
  $\loc = \coprod_{x \in \Adm} \Fl_{x}$
 \end{proposition}   
   
 \begin{remark}
  \emph{(1)} Its left action identifies $\cB$ with the group of automorphisms of $\lambda_\bullet$ that fix $(\,\, ,\,)$ up to a constant. Thus the Schubert cells in $\loc$ coincide with the $\underline{\Aut}(\Lambda_\bullet)$-orbits. \smallskip \\
  \emph{(2)} Let $x\in \widetilde{W}$ and $\bx = (x_0,\ldots,x_{2g-1})$ the corresponding extended alcove. A short calculation shows that $x\cdot\lambda_{-i} = \langle  t^{x_{2g-i} (j) +1} \cdot e_j;\, j=1,\ldots,2g \rangle_{\FF\pot{t}}$ thus the corresponding subspace $\cF_{-i} \subset \Lambda_{-i}$ has basis $\{e_j^{i};\, (x_{2g-i} - \omega_{2g-i})(j) = 0\}$.
 \end{remark}
 
 \begin{definition}
  For any $x\in\Adm$, we define the KR stratum $\cA_x = \varphi (\psi^{-1} (\Fl_x))$.
 \end{definition}
 
 Now $\psi$ is $\underline{\Aut}(\Lambda_\bullet)$-equivariant, thus $\psi^{-1}(\Fl_x)$ is $\underline{\Aut}(\Lambda_\bullet)$-stable. Since $\varphi$ is a $\underline{\Aut}(\Lambda_\bullet)$-torsor, the property of being locally closed descends to $\cA_x$. We endow the KR strata with the reduced subscheme structure. The properties of KR strata can be deduced from the analogous properties of Schubert cells in a similar manner.
 
 \begin{proposition}[\cite{GN}, \S 4] \label{KR} 
  Let $x\in\Adm$. \smallskip \\
  \emph{(1)} $\overline{\cA_x}= \coprod_{y \leq x} \cA_y$. \smallskip \\
  \emph{(2)} $\cA_x$ is smooth of pure dimension $\cl (x)$.
 \end{proposition}
 
 In particular, the $\cA_x$ form a stratification of $\cA_I$
 
 \subsubsection{The $p$-rank on a KR stratum} From now on let $k$ be an algebraically closed field. The calculation of the $p$-rank on a given KR stratum of Ng\^{o} and Genestier in \cite{GN}, Thm.~4.1 also proves that $x$ determines the kernels of the isogenies of the chains $A_\bullet$ corresponding to an $k$-point of $\cA_x$ up to isomorphism. We give a proof quite similar to theirs using Dieudonn\'e theory.
 
Recall that up to isomorphism there are only three finite group schemes of order $p$ over $k$:
 \begin{itemize}
  \item $\ZZ/p\ZZ$, the constant scheme.
  \item $\alpha_p$, characterized by $\alpha_p(R) = \{u\in R; u^p=0\}$.
  \item $\mu_p$, characterized by $\mu_p(R) = \{u\in R; u^p = 1\}$.
 \end{itemize}
 The corresponding Dieudonn\'e modules are characterized by
 \begin{itemize}
  \item $\DD(\ZZ/p\ZZ) \cong \FF$, $F$ bijective, $V = 0$.
  \item $\DD(\alpha_p) \cong \FF$, $F = V = 0$.
  \item $\DD(\mu_p) \cong \FF$, $F = 0$, $V$ bijective.
 \end{itemize}
 
 \begin{proposition} \label{genestier-ngo}
  Let $\underline{A_\bullet}\in \cA_x(k)$. Denote the kernel of the isogeny $A_{i-1}\rightarrow A_i$ by $K_i$. \smallskip \\
  \emph{(1)} $K_i \cong \mu_p$ iff $w(i) = i, x_0(i) = 1$. \smallskip \\
  \emph{(2)} $K_i \cong \ZZ/p\ZZ$ iff $w(i) = i, x_0(i) = 0$. \smallskip \\
  \emph{(3)} $K_i \cong \alpha_p$ iff $w(i) \not= i$.
 \end{proposition}
 \begin{proof}
  It suffices to prove (1). Then (2) follows by duality and (3) by exclusion. Now the exact sequence of commutative, finite, $p$-torsion group schemes
  \[
   0 \longrightarrow K_i \longrightarrow A_{i-1}[p] \longrightarrow A_i[p]
  \]
  gives rise to an exact sequence of Dieudonn\'e modules
  \[
    \DD(A_i[p]) \stackrel{\alpha}{\longrightarrow} \DD(A_{i-1}[p]) \stackrel{\beta}{\longrightarrow} \DD(K_i) \longrightarrow 0
  \]
  Now $\beta$ restricts to a surjection $V\DD(A_{i-1}[p]) \twoheadrightarrow V\DD(K_i)$, thus we have $V\DD(A_{i-1}[p]) \not\subset \alpha(\DD(A_i[p]))$ if and only if $V\DD(K_i) \not= 0$, i.e.\ $K_i \cong \mu_p$. By Theorem \ref{oda} this translates to the equivalence
  \[
   \omega_{A_{i-1}} \not\subset \alpha(\DR(A_i)) \Lra K_i \cong \mu_p.
  \]
  Choose an isomorphism $\DR(A_\bullet) \cong \Lambda_\bullet$ and let $\{e_1^j,\ldots,e_{2g}^j\}$ be the bases we described in \ref{ssectlc}.
  
  $\cB(k)$ acts on $\DR(A_{i-1})$ by multiplication on the left of block matrices of the form

  \[
   \left( \begin{array}{c|c}
    U_{2g+1-i} & 0 \\ \hline
    M & U_{i-1}
   \end{array} \right)
  \]
  where the $U_r$ are upper triangular $r\times r$ matrices over $k$. Thus $\alpha(\DR(A_i)) = \langle e_j^{i-1}; j \not= 2g+1-i \rangle_\FF$ is $\cB(k)$-stable, so the condition $\omega_{A_{i-1}} \not\subset \alpha(\DR(A_i))$ is invariant under the $\cB (k)$-action. Let $\bx = (x_0,\ldots x_{2g-1})$ denote the extended alcove corresponding to $x$. Then we may assume that $\omega_{A_i} = \langle e_j^{i}; (x_{2g-i} - \omega_{2g-i})(j) = 0 \rangle_k$ (see remark in subsection \ref{ssectflag}). We get
  \begin{equivarray}
   & \omega_{A_{i-1}} \not\subset \alpha(\DR(A_i)) \\
   \Lra & e_{2g+1-i}^{i-1} \in \omega_{A_{i-1}} \\
   \Lra & (x_{2g+1-i} - \omega_{2g+1-i})(2g+1-i) = 0 \\
   \Lra & x_{2g+1-i} (2g+1-i) = -1 \\
   \Lra & x_0 (2g+1-i) = 0 \textnormal{ and } w^{-1}(2g+1-i) \leq 2g+1-i \\
   \Lra & x_0 (i) = 1 \textnormal{ and } w^{-1}(i) \geq i.
  \end{equivarray}
  Now Lemma \ref{mu-adm} implies that the last line is equivalent to $x_0(i) = 1, w(i)= i$.
 \end{proof} 
 
 \begin{corollary}[\cite{GN}, Thm.~4.1] \label{relationship}
  The KR stratification is a refinement of the stratification by $p$-rank. The $p$-rank on the stratum $\cA_x, x \in \Adm$ is given by $\# \{i\in\{ 1, \ldots ,g\}; \, w(i) = i \}$ where $x=t^{x_0} w, w\in W$ and where $w$ is considered as an element of $S_{2g} \supset W$.
 \end{corollary}
 \begin{proof}
  Let $\underline{A_\bullet} \in \cA_x(k)$. Since the multiplication by $p$ is the composition of the maps $A_i \rightarrow A_{i+1}$, its kernel is an extension of the $K_i$. Thus 
  \[
   \log_p (\# A_0[p] (k)) = \sum_{i=0}^{2g-1} \log_p(\# K_i(k)) =\log_p(\frac{1}{2} \#\{i\in\{1,\ldots,2g\}; w(i) = i\})
  \]
 which gives the result.
 \end{proof} 
 
 The following result about KR strata allows us to calculate the dimension of $\cA_I^{(d)}$.
 
 \begin{corollary} \label{lengthdim}
  Denote by $\Adm^{(d)}$ the set of $\mu$-admissible elements, which give rise to a KR-stratum on which the $p$-rank is $d$. Then
  \[
   \dim \cA_I^{(d)} = \max_{x \in \Adm^{(d)}} \cl (x)
  \]
 \end{corollary}
 \begin{proof}
  Since $\cA_I^{(d)}$ is the finite union of locally closed $\cA_x$, we get \enlargethispage{1ex}
  \[
   \dim \cA_I^{(d)} = \max_{x\in\Adm^{(d)}} \dim \cA_x = \max_{x \in\Adm^{(d)}} \cl (x)
  \]
 \end{proof}
 
 We call the a KR-stratum \emph{top-dimensional} if it has the same dimension as the $p$-rank stratum which contains it. Since the KR-strata are equidimensional this is equivalent to saying that all its irreducible components have maximal dimension in the $p$-rank stratum. We call the corresponding $\mu$-admissible elements \emph{of maximal length}. Note that by Corollary \ref{lengthdim} an element $x \in\Adm^{(d)}$ is of maximal length if and only if $\cl(x) = \max_{x'\in\Adm^{(d)}} \cl(x')$.
 
   
 \section{The Dimension of the $p$-rank strata} \label{sect2}

 \subsection{Combinatorics of the symmetric group} In order to estimate the length of $\mu$-admissible elements we need some results from the combinatorics of the symmetric group. The following definition will help us to express the length of certain $\mu$-admissible elements.

 \begin{definition}
  Let $\sigma \in S_g$. We define
  \begin{eqnarray*}
   A_\sigma &=& \# \{(i,j)\in\{1,\ldots , g\}^2 ; \, i<j<\sigma(j)<\sigma(i) \}\\
   B_\sigma &=& \# \{(i,j)\in\{1,\ldots , g\}^2 ; \, i<j=\sigma(j)<\sigma(i) \}\\
   C_\sigma &=& \# \{(i,j)\in\{1,\ldots , g\}^2 ; \, i<j<\sigma(i)<\sigma(j) \}.
  \end{eqnarray*}
 \end{definition}

 The following proposition is a reformulation of a result of Clarke, Steingr\'imsson and Zeng;  using their notation it states that $\INV = \INV_{\MT}$ (\cite{CSZ}, Prop.9).

 \begin{proposition} \label{INV=INVMT}
  Let $\sigma \in S_g$. Then
  \[
   \cl (\sigma) = 2(A_\sigma + A_{\sigma^{-1}}+B_\sigma) + C_\sigma + C_{\sigma^{-1}} + \#\{i;\, i < \sigma (i)\} + \#\{ i ; \, \sigma(\sigma (i)) < \sigma(i) < i \}
  \]
 \end{proposition}

 \begin{proof}
  Since the notation in \cite{CSZ} is entirely different from ours, we give an elementary proof of the proposition. We will reduce the claim to the following lemma.
  \begin{lemma}[\cite{CSZ}, Lemma 8]
   Let $\sigma \in S_g$ be a permutation. Write $a_i = \sigma (i)$. Then
   \begin{eqnarray*}
    \# \{(i,j);\, i\leq j<a_i,a_j>j\} &=& \# \{(i,j);\, a_i<a_j\leq i, a_j > j \} \\
    \# \{(i,j);\, i\leq j<a_i,a_j\leq j \} &=& \# \{(i,j);\, a_i<a_j\leq i, a_j \leq j \}
   \end{eqnarray*}
  \end{lemma}
  The first equation of this lemma is proven in \cite{Clarke}, Lemma 3. Note that the second identity is a consequence of the first.
   
  Now let $\sigma \in S_g$. To shorten the notation we write $a_i = \sigma (i)$. Writing $\cl (\sigma)$ as number of inversions, i.e $\cl (\sigma) = \# \{(i,j);\, i<j, \sigma (i) > \sigma (j)\}$, we get
  \begin{eqnarray*}
   \cl (\sigma) - A_\sigma - A_{\sigma^{-1}} - B_{\sigma^{-1}} &=& \#\{(i,j);\, i<j, a_i >a_j, a_i > i, a_j \leq j\} \\
   &=& \#\{(i,j);\, i<j,a_i>a_j,a_i>i,a_j\leq j,j\geq a_i\}\\ & & + \#\{(i,j);\, i<j, a_i>a_j, a_i>i,a_j\leq j, j < a_i\} \\
   &=& \#\{(i,j);\, a_j<a_i\leq j, a_i > i\} \\ & & + \#\{(i,j);\, i<j<a_i, a_j \leq j\} \\
   &=& \#\{(i,j);\, i\leq j < a_i, a_j > j\} \\ & & + \#\{(i,j);\, a_i<a_j\leq i, a_j \leq j \} \\
   &=& (A_\sigma + C_\sigma + \#\{i;\, i<a_i\}) \\ 
   & & + (A_{\sigma^{-1}} + B_\sigma + C_{\sigma^{-1}}+\#\{i;\, a_{\sigma (i)} < a_i < i\}).
  \end{eqnarray*}
  Since $B_\sigma = B_{\sigma^{-1}}$, the assertion follows.
  \end{proof}

 \begin{corollary} \label{symgroupcombi}
  For every $\sigma \in S_g$ we have
  \begin{equation} \label{sgc}
   \cl (\sigma ) - 2(A_\sigma + A_{\sigma^{-1}} + B_\sigma) \geq \frac{1}{2}\cdot\#\{i;\, \sigma(i) \not=i\}.
  \end{equation}
 \end{corollary}

 \begin{proof}
  As a consequence of Lemma \ref{INV=INVMT} we get
  \[
   \cl (\sigma) - 2(A_\sigma + A_{\sigma^{-1}} + B_\sigma) \geq \#\{i;\, i<\sigma(i)\}.
  \]
  Now the left hand side does not change if we replace $\sigma$ by $\sigma^{-1}$. Hence it is also greater or equal $\#\{i;\, \sigma (i) <i\}$. Thus,
  \[
   \cl (\sigma) - 2(A_\sigma + A_{\sigma^{-1}} + B_\sigma) \geq \frac{1}{2}\cdot  (\#\{i;\, i<\sigma (i)\} + \#\{i;\, \sigma (i)< i\}) = \frac{1}{2}\cdot\#\{i;\, \sigma (i) \not=i\}
  \]
  which gives the desired result.
 \end{proof}

 \subsection{Calculation of $\dim \cA_I^{(d)}$} \label{sdimension}
 As in the preceding sections, we consider $W$ as subgroup of $S_{2g}$. Note that every element of $W$ is uniquely defined by its images of $1,\ldots ,g$.  In particular, so are its fixed points. So whenever we speak of fixed points of a Weyl group element, we only mean those which are smaller or equal to $g$ (unless stated otherwise).

 For any subset $F \subset \{ 1,2, \ldots , g \}$ let $W^F \subset W$ be the subgroup of elements whose fixed points are exactly the elements of $F$. Denote by $\Adm^F$ the preimage of $W^F$ under the canonical projection $pr: \Adm \to W,\, t^{x_0}\cdot w \mapsto w$.  Then by Proposition \ref{relationship}
 \[
  \Adm^{(d)} = \bigcup_{F \subset \{1,\ldots,g\}, \atop \# F = d} \Adm^F.
 \]
 
 We fix a non-negative integer $d \leq g$ and a set $F = \{f_1 < f_2 < \ldots < f_d\} \subset \{ 1,2,\ldots ,g\}$. Using Lemma \ref{mu-adm} we may index the preimage of any element in $W^F$ by $\{ 0,1 \}^d$. For every $w\in W^F$ we have $pr^{-1} (w) = \{t^{x_{w,v,0}}\cdot w;\, v \in \{0,1\}^d\}$ where
 \[
  x_{w,v,0}(i) = \left\{
  \begin{array}{ll}
   0 & \textnormal{if } w^{-1}(i) > i \\
   1 & \textnormal{if } w^{-1}(i) < i \\
   v(j) & \textnormal{if } i = f_j \\
   1-v(j) & \textnormal{if } i = 2g+1-f_j
  \end{array}
  \right.
 \]
 For $v\in\{0,1\}^d$ denote by $\Adm^{F,v}$ the subset of $\Adm^F$ of elements that are indexed by $v$. Then
 \[
  \Adm^F = \bigcup_{v \in \{0,1\}^d} \Adm^{F,v}
 \]
 and we obtain a one-to-one correspondence
 \begin{eqnarray*}
  W^F &\rightarrow& \Adm^{F,v} \\
  w &\mapsto& t^{x_{w,v,0}} \cdot w,
 \end{eqnarray*}
 i.e.\ an element of $\Adm^{F,v}$ is uniquely determined by its Weyl group - component.
 
 We fix a vector $v \in \{0,1\}^d$. To simplify the notation we establish the following notation: For integers $1 \leq i,j \leq g$ let
 \begin{eqnarray*}
  i \prec j &:\Lra& i<j \textnormal{ or } (i=j=f_k \textnormal{ and } v(k) = 1) \\
  i \succ j &:\Lra& i>j \textnormal{ or } (i=j=f_k \textnormal{ and } v(k) = 0) \\
  i \approx j &:\Lra& i = j \textnormal{ and } i \not\in F \\
  i \preccurlyeq j &:\Lra& i \prec j \textnormal{ or } i \approx j \\
  i \succcurlyeq j &:\Lra& i \succ j \textnormal{ or } i \approx j
 \end{eqnarray*}

 \begin{remark}
  There is also a geometrical interpretation of fixing $F$ and $v$. According to Proposition \ref{genestier-ngo} this is the same as fixing the $K_i$ up to isomorphism.
 \end{remark}

 We now identify $W$ with $\FF_2^g \rtimes S_g$, an element $\sigma \in S_g$ corresponds to a permutation $w \in W$ with $w(i) = \sigma(i),\, i=1,\ldots g$ and a vector $u \in \FF_2^g$ corresponds to an element $w$ with $w(i)=i$ if $u(i)=0$ and $w(i)=2g+1-i$ if $u(i) =1,\, i = 1,\ldots, g$. (Recall that the elements of $W$ are uniquely determined by their restriction to $\{ 1,\ldots , g\}$.)
 Denote by $S_g^F$ the subset of all $\sigma\in S_g$ for which there exists a $u\in\FF_2^g$ such that $u\sigma \in W^F$. Obviously it is the set of all permutations which fix the elements of $F$ . Let $\sigma \in S_g^F$. A vector $u \in \FF_2^g$ is called $\sigma$-admissible if $u\sigma \in W^F$, which is equivalent to
 \begin{myeqnarray}
  u(i) &=& 0 & \forall i \in F \\
  u(i) &=& 1 & \forall i : \sigma(i) \approx i.
 \end{myeqnarray}
 So if $i$ is a fixed point of $\sigma$ then $u(i)$ has a preset value, but we can choose the other components freely. However, this changes when we consider the maximal elements among them.

 \begin{lemma} \label{GY82}
  Let $\sigma \in S_g^F$, $u,u' \in \FF_2^g$ $\sigma$-admissible with $u' \leq u$ (Imposing the canonical order on $\FF_2$ = \{0,1\}). Then $t^{x_{(u'\sigma),v,0}}(u'\sigma)$ dominates $t^{x_{(u\sigma),v,0}}(u\sigma)$ w.r.t.\ the Bruhat order.
 \end{lemma}
 \begin{proof}
  By induction, it suffices to prove the assertion in the case when $u$ and $u'$ differ in exactly one, say the $i$-th, component. By the observation we made above, this implies that $\sigma (i) \not= i$. The rest of the proof is the same as the proof of Lemma 8.2 in \cite{GY2}.
 \end{proof}

 Let us agree on the following convention. Whenever we write an element $x \in \widetilde{W}$ as a product like $x = t^{x_0}\cdot w$ or $x=t^{x_0} \cdot (u\sigma)$, we assume that $x_0 \in X_*(T), u \in \FF_2^{g},\sigma \in S_{g}$ and $w \in W$. We call $x_0$, $w$, $u$ resp.\ $\sigma$ its $X_* (T)$-, $W$-, $\FF_2^g$- resp.\ $S_g$-component.

 As a consequence of Lemma \ref{GY82} and the description of $\Adm^{F,v}$ at the beginning of this section we get the following assertion.
 
 \begin{corollary} \label{maxelt}
  Assume that $t^{x_0}\cdot (u\sigma)$ is a maximal element with respect to the Bruhat order in $\Adm^{F,v}$. Then for all $1\leq i \leq g$
  \begin{eqnarray*}
   u(i) &=& \left\{ 
   \begin{array}{ll} 
    1  & \textnormal{ if } \sigma (i) \approx i \\
    0  & \textnormal{ if } \sigma (i) \not\approx i
   \end{array} \right. \\ 
   x_0 (i) &=& \left\{
   \begin{array}{ll}
    1 & \textnormal{ if } \sigma^{-1} (i) \prec i \\
    0 & \textnormal{ if } \sigma^{-1} (i) \succcurlyeq i
   \end{array} \right.
  \end{eqnarray*}
  More precisely any element $t^{\widetilde{x}_0}(\widetilde{u}\sigma)$ is dominated by $t^{x_0} (u\sigma)$ where $x_0$ and $u$ are given by above description.
 \end{corollary}
 
 \begin{definition} \label{possiblymaximal}
  We call a $\mu$-admissible element possibly maximal if it satisfies the condition of Corollary \ref{maxelt}.
 \end{definition}

 The following proposition is the core of this paper, it enables us to calculate the dimension of the $p$-rank strata and to determine the top-dimensional KR-strata.

 \begin{proposition} \label{main}
  Let $x=t^{x_0}(u\sigma)$ be possibly maximal. \smallskip \\
  \emph{(1)} For all $\sigma\in S_g^F$ we have
  \begin{eqnarray*}
   \cl (w) &=& \frac{g(g+1)}{2} +d - \# \{i;\sigma(i) = i\} -\cl (\sigma) \\ & & + 2 (A_\sigma + A_{\sigma^{-1}} + \# \{(i,f) \in\{1,\ldots ,g\}\times F;\, i<f<\sigma(i)\}).
  \end{eqnarray*}
  \emph{(2)} In particular, for all $\sigma$
  \[
   \cl (x) \leq \left\lfloor \frac{g^2+d}{2} \right\rfloor .
  \]
  \emph{(3)} Let $F=\{g-d+1,g-d+2,\ldots g\}$ and $\sigma = (1\enspace 2)(3\enspace 4)\ldots (g-d-1\enspace g-d)$ if $g-d$ is even resp. $\sigma = (1\enspace 2)(3\enspace 4)\ldots (g-d-2\enspace g-d-1)$ if $g-d$ is odd. Then
  \[
   \cl (x) = \left\lfloor \frac{g^2+d}{2} \right\rfloor .
  \]
 \end{proposition}

 \begin{proof}
  \emph{(1)} We calculate $\cl (t^{x_0}(u\sigma))$ using Lemma \ref{IMF}. Recall that
  \[
   \cl (w) = \sum_{\beta \in R^+ \atop w^{-1}\beta \in R^+ } |\langle \beta,x_0 \rangle | + \sum_{\beta \in R^+ \atop w^{-1}\beta \not\in R^+} |\langle \beta, x_0 \rangle + 1|
  \]
  In order to calculate the sum on the right hand side, we divide it up into ten smaller sums. In the process we distinguish between the cases
 \begin{enumerate}
  \item[(a)] $\beta \in \{ \beta_{ij}^1; 1\leq i < j \leq g \}$
  \item[(b)] $\beta \in \{ \beta_{ij}^2; 1\leq i < j \leq g \}$
  \item[(c)] $\beta \in \{ \beta_{i}^3; 1 \leq i \leq g \}$
 \end{enumerate}
 and the values of $\{ u(i), u(j) \} \in \FF_2^2$ when we consider the cases (a) and (b) resp. the value of $u(i) \in \FF_2^2$ when we consider the case (c). We denote these sums by $\Sigma_{a, (u(i),u(j))}$, $\Sigma_{b, (u(i), u(j))}$ and $\Sigma_{c, u(i)}$.
  
  (a) \emph{Summands coming from $\beta_{i,j}^1$.} Let $1\leq i<j \leq g$. We need to check when $w^{-1}\beta_{ij}^1$ is positive and calculate $\langle \beta_{ij}^1, x_0 \rangle$. Using the description of positive roots given in section 1.1 we get
  \begin{equation} \label{2.9.a}
   w^{-1}\beta_{ij}^1 > 0 \Lra u(i) = 0 \textnormal{ and } (u(j) = 1 \textnormal{ or } \sigma^{-1} (i) < \sigma^{-1} (j))
  \end{equation}
  \[
   \langle \beta_{ij}^1, x_0 \rangle = x_0(i) - x_0(j) = \left\{ 
   \begin{array}{rl}
    1 & \textnormal{ if } \sigma^{-1}(i)\prec i, \sigma^{-1}(j)\succcurlyeq j \\
    0 & \textnormal{ if } \sigma^{-1}(i)\prec i, \sigma^{-1}(j)\prec j \textnormal{ or } \sigma^{-1}(i)\succcurlyeq i, \sigma^{-1}(j)\succcurlyeq j \\
    -1 & \textnormal{ if } \sigma^{-1}(i)\succcurlyeq i, \sigma^{-1}(j)\prec j
   \end{array}
   \right. 
  \]  
  We will give the calculation of $\Sigma_{a,(0,0)}$ in full detail - the other sums are calculated analogously.
  \begin{eqnarray*}
   \Sigma_{a,(0,0)} &=& \sum_{1\leq i < j \leq g \atop u(i) = u(j) = 0 } \left( \sum_{w^{-1} \beta_{ij}^1 \in R^+} |\langle \beta_{ij}^1,x_0\rangle | + \sum_{w^{-1} \beta_{ij}^1 \not\in R^+} |\langle \beta_{ij}^1,x_0\rangle + 1| \right) \\
   &=& \sum_{1\leq i < j \leq g \atop \sigma (i) \not\approx i, \sigma(j) \not\approx j} \left( \sum_{\sigma^{-1}(i) <\sigma^{-1}(j)} |\langle \beta_{ij}^1, x_0 \rangle | + \sum_{\sigma^{-1}(i) > \sigma^{-1}(j)} |\langle \beta_{ij}^1,x_0\rangle + 1|\right) \\ 
  \end{eqnarray*}
  Here the second equality holds because by definition of ``possibly maximal'', $u(i) = 0$ is equivalent to $\sigma(i) \not\approx i$ and by (\ref{2.9.a}) we have equivalence between $w^{-1}\beta_{ij}^1 \in R^+$ and $ \sigma^{-1} (i) < \sigma^{-1} (j)$ in the case $u(i) = u(j) = 0$. Our calculation of $\langle \beta_{ij}^1, x_0\rangle$ implies that the first sum equals (we always sum over $i,j \in \{1,\ldots ,g\}$ unless stated otherwise)
  \begin{eqnarray}
   & & \# \{ (i,j);\, i<j, \sigma^{-1} (i) \prec i, \sigma^{-1} (j) \succ j, \sigma^{-1} (i) < \sigma^{-1} (j) \} \label{1} \\ 
   & & + \# \{ (i,j);\, i<j, \sigma^{-1} (i) \succ i, \sigma^{-1} (j) \prec j, \sigma^{-1} (i) < \sigma^{-1} (j) \} \label{2}
  \end{eqnarray}
  and that second sum equals
  \begin{eqnarray}
   & & \# \{ (i,j);\, i<j, \sigma^{-1} (i) \prec i, \sigma^{-1} (j) \prec j, \sigma^{-1} (i) > \sigma^{-1} (j) \} \label{3} \\
   & & + \# \{ (i,j);\, i<j, \sigma^{-1} (i) \succ i, \sigma^{-1} (j) \succ j, \sigma^{-1} (i) > \sigma^{-1} (j) \}. \label{4}
  \end{eqnarray}
  Using the same arguments we get 
  \begin{eqnarray}
   \Sigma_{a,(0,1)} &=& \# \{ (i, j);\, i<j, \sigma^{-1} (i) \prec i, \sigma^{-1} (j) \approx j \} \label{5} \\
   \Sigma_{a,(1,0)} &=& \# \{ (i, j);\, i<j, \sigma^{-1} (i) \approx i, \sigma^{-1} (j) \succ j \} \label{6} \\
   \Sigma_{a,(1,1)} &=& \# \{ (i, j);\, i<j, \sigma^{-1} (i) \approx i, \sigma^{-1} (j) \approx j \}. \label{7}
  \end{eqnarray}
  Altogether, $\sum_{i,j\in\FF_2} \Sigma_{a, (i,j)}$ equals 
  \begin{eqnarray}
   & & \# \{ (i, j);\, i<j, \sigma^{-1} (i) \preccurlyeq i, \sigma^{-1} (j) \succcurlyeq j, \sigma^{-1} (i) < \sigma^{-1} (j) \} \label{01}\\
   & & + \# \{ (i,j);\, i<j, \sigma^{-1} (i) \succ i, \sigma^{-1} (j) \prec j, \sigma^{-1} (i) < \sigma^{-1} (j) \} \label{02}\\
   & & + \# \{ (i,j);\, \sigma^{-1} (i) > \sigma^{-1} (j) \succ j > i \} \label{03} \\
   & & + \# \{ (i,j);\, \sigma^{-1} (j) < \sigma^{-1} (i) \prec i < j \}. \label{04}
  \end{eqnarray}
  Let us impose that until the end of this proof whenever we refer to a equation we only mean the right hand side. Then (\ref{01}) is the sum of  (\ref{1}), (\ref{5}), (\ref{6}) and (\ref{7}) while (\ref{02}), (\ref{03}) and (\ref{04}) are just reformulations of (\ref{2}), (\ref{3}) and (\ref{4}), respectively.
  
  (b) \emph{Summands coming from $\beta_{ij}^2$.} We proceed as in the first case.
  \begin{eqnarray*}
   w^{-1} \beta_{ij}^2 > 0 &\Lra& u(i)=u(j)=0 \textnormal{ or } (u(i)=0 \textnormal{ and } \sigma^{-1} (i) < \sigma^{-1} (j)) \\ & & \textnormal{ or } (u(j)=0 \textnormal{ and } \sigma^{-1} (i) > \sigma^{-1} (j))
  \end{eqnarray*}
  \[
   \langle \beta_{ij}^2 , x_0 \rangle = x_0 (i) + x_0(j) - 1 = \left\{
   \begin{array}{rl}
    1 & \textnormal{ if } \sigma^{-1}(i) \prec i, \sigma^{-1}(j) \prec j \\
    0 & \textnormal{ if } \sigma^{-1}(i) \prec i, \sigma^{-1}(j) \succcurlyeq j \textnormal{ or } \sigma^{-1}(i) \succcurlyeq i, \sigma^{-1}(j) \prec j\\
    -1 & \textnormal{ if } \sigma^{-1}(i) \succcurlyeq i, \sigma^{-1}(j) \succcurlyeq j
   \end{array} \right.
  \] 
  We conclude
  \begin{eqnarray}
   \Sigma_{b, (0,0)} &=& \# \{ (i,j);\, i<j, \sigma^{-1} (i) \prec i, \sigma^{-1} (j) \prec j, \sigma^{-1} (i) < \sigma^{-1} (j) \} \label{8}\\
   & & + \# \{ (i,j);\, i<j, \sigma^{-1} (i) \prec i, \sigma^{-1} (j) \prec j, \sigma^{-1} (i) > \sigma^{-1} (j) \} \label{9}\\
   & & + \# \{ (i,j);\, i<j, \sigma^{-1} (i) \succ i, \sigma^{-1} (j) \succ j, \sigma^{-1} (i) < \sigma^{-1} (j) \} \label{10}\\
   & & + \# \{ (i,j);\, i<j, \sigma^{-1} (i) \succ i, \sigma^{-1} (j) \succ j, \sigma^{-1} (i) > \sigma^{-1} (j) \} \label{11}\\
   \Sigma_{b,(0,1)} &=& \# \{ (i.j);\, i<j, \sigma^{-1} (i) \succ i, \sigma^{-1} (j) \approx j, \sigma^{-1} (i) < \sigma^{-1} (j) \} \label{12} \\
   \Sigma_{b,(1,0)} &=& \# \{ (i,j);\, i<j, \sigma^{-1} (i) \approx i, \sigma^{-1} (j) \prec j, \sigma^{-1} (i) < \sigma^{-1} (j) \} \label{13} \\
   \Sigma_{b,(1,1)} &=& 0. \nonumber
  \end{eqnarray}
  Altogether, $\sum_{i,j\in\FF_2} (\Sigma_{a, (i,j)} + \Sigma_{b,(i,j)})$ equals
  \begin{eqnarray}
   & & \# \{ (i,j);\, i<j,  \sigma^{-1} (i) < \sigma^{-1} (j) \} \label{05} \\
   & & +2\cdot \# \{ (i,j);\, \sigma^{-1} (i) > \sigma^{-1} (j) \succ j > i \} \label{06} \\
   & & +2\cdot \# \{ (i,j);\, \sigma^{-1} (j) < \sigma^{-1} (i) \prec i < j \}. \label{07}
  \end{eqnarray}
  Here (\ref{05}) is the sum of (\ref{01}), (\ref{02}), (\ref{8}), (\ref{10}), (\ref{12}) and (\ref{13}); (\ref{06}) is the sum of (\ref{03}) and (\ref{11}); and (\ref{07}) is the sum of (\ref{04}) and (\ref{9}).
  
  (c) \emph{Summands coming from $\beta_i^3$}
  \[
   w^{-1}\beta_i^3 > 0 \Lra u(i) = 0
  \]
  \[
   \langle \beta_i^3 , x_0 \rangle = 2x_0(i)-1 = \left\{ 
   \begin{array}{lr}
    1 & \textnormal{if } \sigma^{-1}(i) \prec i \\
    -1 & \textnormal{if } \sigma^{-1}(i) \succcurlyeq i
   \end{array}
   \right.
  \]
  
  Thus we get
  \begin{eqnarray*}
   \Sigma_{c,0} &=& \# \{ i;\, \sigma^{-1} (i) \prec i \} +\#\{ i;\, \sigma^{-1} (i) \succ i \} \\
   &=& \# \{i; \sigma^{-1} (i) \not\approx i\} \\
   \Sigma_{c,1} &=& 0.
  \end{eqnarray*}
    
  So the final sum is 
  \begin{eqnarray*}
   \cl (t^{x_0} w) &=& \# \{ i;\, \sigma(i) \not\approx i \} \\ & & + \# \{(i,j);\,i<j, \sigma^{-1} (i) < \sigma^{-1} (j) \} \\ & & +2\cdot \#\{(i,j);\, \sigma^{-1} (i) > \sigma^{-1} (j) \succ j > i \} \\ & & + 2\cdot \# \{(i,j);\, \sigma^{-1} (j) < \sigma^{-1} (i) \prec i < j\} \\
    &=& g + d -\# \{i;\, \sigma(i) = i\}  \\ & & +\frac{g(g-1)}{2} - \cl (\sigma) \\ & &  +2\cdot(A_\sigma + \{(j,k);\, \sigma^{-1} (j) < f_k < j, v(k)=1\}) \\ & & + 2\cdot (A_{\sigma^{-1}} + \{(i,k);\, i < f_k < \sigma^{-1} (i) , v(k) = 0\}) \\
    &=& \frac{g(g+1)}{2} +d - \# \{i;\sigma(i) = i\} -\cl (\sigma) \\ & & + 2 (A_\sigma + A_{\sigma^{-1}} + \# \{(i,k);\, i<f_k<\sigma(i)\})
  \end{eqnarray*}
  where the sum is taken over $i,j \in \{1,\ldots ,g\}$ and $k \in \{1, \ldots ,d\}$.
   
  \emph{(2)} The claim is a consequence of part (1) of the proposition and Corollary \ref{symgroupcombi}. Since $\# \{(i,k);\, i<f_k<\sigma (i) \} \leq B_\sigma$ we get 
  \begin{eqnarray*}
   \cl (t^{x_0} w) &\leq& \frac{g(g+1)}{2} + d - \#\{i;\, \sigma(i) = i \} - \frac{1}{2}\cdot(g - \# \{ i;\, \sigma (i) = i \}) \\
   &=& \frac{g^2}{2} + d - \frac{1}{2} \cdot \# \{i; \sigma (i) = i\} \\
   &\leq& \frac{g^2}{2} + d - \frac{d}{2} \\
   &=& \frac{g^2+d}{2}.
  \end{eqnarray*}
  Now $\cl (t^{x_0} w)$ is an integer, hence the claim follows.
   
  \emph{(3)} We apply the formula of part (1). Now
  \begin{eqnarray*}
   \cl (\sigma) &=& \left\lfloor \frac{g-d}{2} \right\rfloor \\
   A_\sigma = A_{\sigma^{-1}} &=& 0 \\
   \# \{(i,k);\, i < f_k < \sigma (i) \} &=&0 \\
   \# \{i;\, \sigma (i) = 0\} &=& d + \delta_{2 \nmid g-d}
  \end{eqnarray*}
  Hence
  \begin{eqnarray*}
   \cl (t^{x_0}\sigma) &=& \frac{g(g+1)}{2} + d - (d + \delta_{2 \nmid g-d}) - \frac{g-d-\delta_{2 \nmid g-d}}{2} \\
   &=& \frac{g^2+d - \delta_{2 \nmid g-d}}{2} \\
   &=& \left\lfloor \frac{g^2+d}{2} \right\rfloor
  \end{eqnarray*}
 \end{proof}

 \begin{remark}
 Let $x=t^{x_0}\cdot w$ be possibly maximal. Denote by $x'$ the preimage of $x$ w.r.t.\ the homomorphism $\widetilde{W}_{\GSp_{2g-2d}} \hookrightarrow \widetilde{W}$ induced by the embedding $\GSp_{2g-2d} \hookrightarrow \GSp_{2g}$ which adds the $k$-th rows and columns $(0^{(k-1)}\,1\,0^{(2g-k)})$ for all $k \in F$ and $2g+1-k \in F$. This amounts to removing the $f$-th and $2g+1-f$-th coordinates of the $X_* (T)$-component and skipping the fixed points of $w$ (here we also mean those which are greater than $g$). Proposition \ref{main}\teil{1} implies that
 \[
  \cl (x) - \cl(x') = {g+1 \choose 2} - {g-d+1 \choose 2}.
 \]
 In particular the difference is independent of $x$. So the length of $x$ does not depend on $v$ and to some extent it also does not depend on $F$.
 \end{remark}
 
 As a consequence of Corollary \ref{lengthdim} and part 2 and 3 of Proposition \ref{main} we obtain the first main result of this paper, which is part (1) of Theorem \ref{results} in the introduction.
 \begin{theorem} \label{dim}
  The dimension of $\cA_I^{(d)}$ is $\left\lfloor \frac{g^2+d}{2} \right\rfloor$.
 \end{theorem}

 \section{Comparison between the KR stratification and $p$-rank stratification} \label{sect3}
 We are going to prove the statement of Theorem \ref{results}\teil{2} next. Note that
 \[
  \overline{\cA_I^{(d)}} = \bigcup_{y\in\Adm^{(d)}} \overline{\cA_y} = \bigcup_{y\in\Adm^{(d)}}\bigcup_{x\leq y} \cA_x
 \]
 where the last equality holds because of Proposition \ref{KR}\teil{1}. So our strategy will be examining for every $\mu$-admissible $x$ the set of $\mu$-admissible $y$ which dominate $x$ w.r.t. the Bruhat order, respectively the $p$-rank on the KR strata associated to these $y$.  
 
 \subsection{Going up in $p$-rank} \label{sectgoingup}
 In this section we prove some first lemmas concerning elements dominating a given $\mu$-admissible element $x$. The main part of this section is the construction of a set of examples, which we will use most of the time when we want to construct dominating elements of given $p$-rank.
 
 Since the $p$-rank is no longer fixed, we also cannot assume $F$ and $v$ to be fixed any longer. Thus we have to introduce some more notation. Denote the set of fixed points of the $W$-component of an element $x\in\Adm$ by $\F (x)$. We also have to reformulate the notion of possibly maximal (cf.\ Definition \ref{possiblymaximal}). We call $x = t^{x_0} (u\sigma) \in \Adm$ possibly maximal if for every $i = 1,\ldots,g$
 \[
  u(i) = 1 \Rightarrow \sigma(i) = i.
 \]
 It is easy to see that this definition is equivalent the definition we gave earlier. Thus we can still use Corollary \ref{maxelt} implying that every $x\in\Adm$ is dominated by a possibly maximal element $y$ which has the same $S_g$-component such that $\F(x) = \F(y)$. In particular if $x\in\Adm^{(d)}$ then we also have $y\in\Adm^{(d)}$.
 
 \begin{lemma} \label{prank1}
   $\overline{\cA_I^{(d)}} \subseteq \bigcup_{d'\leq d} \cA_I^{(d')}$  
 \end{lemma}
 \begin{proof}
  By Theorem \ref{koblitz}\teil{1}, the right hand side equals $\pi^{-1} (\overline{\cA_g^{(d)}})$ which is a closed set containing $\cA_I^{(d)}$.
 \end{proof}

 \begin{corollary} \label{lastminute}
  Let $0 \leq d < d' \leq g$, $x\in\Adm^{(d)}$ and $y\in\Adm^{(d')}$. Then $x \leq y$ if and only if $x$ and $y$ are related with respect to the Bruhat order.
 \end{corollary}
 \begin{proof}
  This is an easy consequence of Lemma \ref{prank1} and Proposition \ref{KR}\teil{1}.
 \end{proof}

 \begin{definition}  
  Let $w=u\sigma \in W$ and $\sigma = Z_1 \cdots Z_n$ the decomposition into disjoint cycles (in the usual sense of the cycle decomposition in the symmetric group, including the cycles of order 1). We say that $Z$ is a cycle of $w$ if $Z\in\{Z_1,\ldots,Z_n\}$ such that either $\ord Z \geq 2$ or if $\ord Z = 1$, i.e. $Z=(i)$ where $i$ is a fixed point of $\sigma$, we have that $i$ is not a fixed point of $w$. If $x\in\Adm$ has $W$-component $w$ we also call $Z$ a cycle of $x$. The set of cycles of $x$ is denoted by $\Z(x)$.
 \end{definition}
 
 By construction  $\Z(x)$ is uniquely determined by $\F(x)$ and the $S_g$-component $\sigma$ of $x$. The converse is also true: Since $\Z(x)$ determines the cyclic decomposition of $\sigma$, it also determines $\sigma$. Now an integer $1\leq i \leq g$ occurs in a (unique) cycle of $x$ if and only if $i \not\in \F(x)$, i.e.\ $\F(x)$ is the set of all integers $1\leq i \leq g$ which do not occur in a cycle of $\Z(x)$. 
 
 As a consequence we may reformulate our conclusion of Corollary \ref{maxelt} as follows. For every $\mu$-admissible $x$ there exists a possibly maximal element $y\in\Adm$ such that $x\leq y$ and $\Z(x) = \Z(y)$.
 
 Also note that our new description of $\F(x)$ yields that for any $x\in\Adm^{(d)}$ 
 \[
  \sum_{Z\in\Z(x)} \ord Z = g-d.
 \]

 We simplify the notation for the coming calculations. Denote by $s_i,s_{ij} \in W$ the reflections corresponding to the roots $\beta_i^3$ resp.\ $\beta_{ij}^1$. With respect to the identification $W=\FF_2^g \rtimes S_g$ we have
 \begin{eqnarray*}
  s_i &=& (0^{(i-1)}\, 1 \, 0^{(g-i)})\cdot \id \\
  s_{i,j} &=& (0^{(g)}) \cdot (i \enspace j)
 \end{eqnarray*} 
 
 \begin{lemma}\label{going up}
  Let $x = t^{x_0} w = t^{x_0} (u\sigma) \in \Adm$ be possibly maximal, $Z = (c_1 \,\, \cdots \,\, c_l) \in \Z(x)$. Then there exists a possibly maximal $y\in\Adm$ with $x\leq y$ satisfying the following additional assumption depending on $Z$. \\ \smallskip
   \emph{(1)} If $\ord Z = 1$ then 
   \begin{eqnarray*}
    \F(y) &=& \F(x) \cup \{c_1\} \\
    \Z(y) &=& \Z(x) \setminus \{Z\}. \quad\,\,
   \end{eqnarray*}
   \emph{(2)} If $\ord Z = 2$ then 
   \begin{eqnarray*}
   \F(y) &=& \F(x) \cup \{c_1, c_2\} \\
   \Z(y) &=& \Z(x) \setminus \{Z\}
   \end{eqnarray*}
   \emph{(3)} If $\ord Z \geq 3$ then there exists an integer $1\leq k \leq l$ such that 
   \begin{eqnarray*}
   \quad\,\,\, \F(y) &=& \F(x) \cup \{c_k\} \\
   \Z(y) &=& \Z(x) \setminus \{Z\} \cup \{Z'\}
   \end{eqnarray*}
    where $Z' = (c_1 \,\, \cdots\,\, c_{k-1}\enspace c_{k+1}\,\, \cdots\,\,  c_m)$.
 \end{lemma}
 \begin{proof}
 Surely the condition on $\F(y)$ follows from the condition on  $\Z(y)$. Now observe that due to Corollary \ref{lastminute} we can replace the demand that $y$ dominates $x$ with respect to the Bruhat order by the weaker demand that they are related. For this it is sufficient that $y = s\cdot x$ or equivalently $y = x\cdot s$ where $s$ is an affine reflection.
 
 In each of the three cases the fact that $y$ is possibly maximal will be an easy consequence of the fact that $x$ is possibly maximal; its verification will therefore be omitted. 
 
 \emph{(1)} Consider $y = s_{c_1}x$. Since $\sigma (c_1) = c_1, u(c_1) = 1$, we get that $c_1$ is a fixed point of $y$ implying the condition on $\Z(y)$. Now by Theorem \ref{KR453} an element of the extended affine Weyl group is $\mu$-admissible if and only if the corresponding $G$-alcove is minuscule. Therefore the criterion of Lemma \ref{mu-adm} can also be used to check whether an element of $\widetilde{W}$ is $\mu$-admissible. In our case the fact that $x$ meets this criterion obviously implies that $y$ satisfies it; thus $y\in\Adm$.
 
 \emph{(2)} Let $y = s_{c_1,c_2}\cdot x$. We have $\sigma (c_1) = c_2, \sigma (c_2) = c_1$ (and thus $u(c_1) = u(c_2) = 0$ since $x$ is possibly maximal). Therefore we get that  $\F (y) = \F(x) \cup \{c_1, c_2\}$ and $\Z(y) = \Z(x) \setminus \{Z\}$. Using the same argumentation as in (1), we see that the fact that $y$ is $\mu$-admissible follows from Lemma \ref{mu-adm}.
 
 \emph{(3)} Since $l \geq 3$, there exists a ``local minimum''   $c_{k-1} > c_k < c_{k+1}$ with $c_{k-1} \not= c_{k+1}$ where $c_0 := c_l$ and $ c_{l+1} := c_1$. If $c_{k-1} > c_{k+1}$ let $y = s_{c_k c_{k+1}} x$, otherwise let $y = x s_{c_k,c_{k-1}}$. One easily checks that the conditions on $\F(y)$ and $\Z(y)$ are met, so that we only have to prove that $y$ is $\mu$-admissible.
 
  In the case $c_{k-1} > c_{k+1}$ let $y=t^{y_0}w'$. Note that since $x$ is possibly maximal we have $w^{-1} (c_i) = \sigma^{-1} (c_i) = c_{i-1}$ for all $i$. To prove that $y$ is $\mu$-admissible it suffices to check the criterion of Lemma \ref{mu-adm} for the coordinates where the $X_*(T)$-component or the preimage of the $W$-component have changed, i.e. at $c_k$ and $c_{k+1}$. Since $c_k$ is a fixed point of $w'$ there is nothing to check here. Thus
  \begin{eqnarray*}
   w'^{-1} (c_{k+1}) &=& c_{k-1} > c_{k+1} \\
   y_0 (c_{k+1}) &=& x_0 (c_k) = 0
  \end{eqnarray*}
  implies that $y\in\Adm$.
  
  In the case $c_{k-1} < c_{k+1}$, we have $y=t^{x_0}w'$ with the same $w'$ as in the previous case. Thus it again suffices to check the criterion of Lemma \ref{mu-adm} at the coordinate $c_{k+1}$. Indeed,
  \begin{eqnarray*}
   w'^{-1}(c_{k+1}) &=& c_{k-1} < c_{k+1} \\
   y_0 (c_{k+1}) &=& x_0 (c_{k+1}) = 1
  \end{eqnarray*}
 \end{proof}

 \begin{corollary} \label{cor going up}
  Let $x = t^{x_0} w = t^{x_0} (u\sigma)$ be an element of $\Adm^{(d)}$ such that $\sigma$ is not a product of $\frac{g-d}{2}$ disjoint transpositions. Then there exists a $y \in \Adm^{(d+1)}$ such that $x \leq y$ with respect to the Bruhat order.
 \end{corollary}
 \begin{proof}
  Since $x$ is dominated by possible maximal element with the same cycles, we may assume without loss of generality that $x$ is possibly maximal. Since the assertion on $\sigma$ is equivalent to postulate that $x$ at least one cycle of order unequal to $2$, the assertion follows from Lemma \ref{going up}\teil{1} or \ref{going up}\teil{3}.
 \end{proof}
 
 \subsection{Top-dimensional KR strata} \label{secttopdim}
  In this section we give an explicit description of the elements of maximal length, which by definition correspond to the top-dimensional KR strata. As an immediate result of the previous section we get:
 \begin{lemma} \label{top-dim1}
  Let $x\in\Adm^{(d)}$ be of maximal length and $\sigma$ its $S_g$-component. Then $x$ is obviously possibly maximal. \smallskip \\
  \emph{(1)} If $g-d$ is even, $\sigma$ is the product of $\frac{g-d}{2}$ pairwise disjoint transpositions. \smallskip \\
  \emph{(2)} If $g-d$ is odd, $\sigma$ is either the product of $\lfloor\frac{g-d}{2}\rfloor$ disjoint transpositions or the product of $\lfloor \frac{g-d}{2}-1\rfloor$ transpositions and a cycle of order 3 which are pairwise disjoint.
 \end{lemma}
 
 \begin{proof}
 \emph{(1)} If $\sigma$ was not the product of $\frac{g-d}{2}$ pairwise disjoint transpositions, Corollary \ref{cor going up} would give us a $y \in \Adm^{(d+1)}$ with $\cl (y) \geq \cl(x)$. Contradiction to Theorem \ref{dim}.
 
 \emph{(2)} Since $g-d$ is odd we can either use Lemma \ref{going up}\teil{1} or \ref{going up}\teil{3} on $x$. Using the dimension formula once again, we see that this gives us an element $y\in\Adm^{(d+1)}$ of maximal length. Thus the $S_g$-component of $y$ is the product of $\frac{g-d}{2}$ pairwise disjoint transpositions. In the case where we used \ref{going up}\teil{1} this implies that $\sigma$ is the product of $\lfloor \frac{g-d}{2} \rfloor$ pairwise disjoint transpositions, if we used \ref{going up}\teil{3} $\sigma$ is a product of $\lfloor \frac{g-d}{2} \rfloor -1$ transpositions and a cycle of order 3, which are pairwise disjoint.
 \end{proof}
 
 So we may reduce our considerations to $x \in \Adm$ respectively $\sigma \in S_g$ which are of the form described in Lemma \ref{top-dim1}. Next we check when the left hand side and the right hand side of (\ref{sgc}) are (almost) equal.

 \begin{lemma} \label{sgc_even}
  If $\sigma \in S_g$ is an involution, equality occurs in (\ref{sgc}) if and only if $C_\sigma = 0$.
 \end{lemma}

 \begin{proof}
  The proof is straightforward:
  \begin{eqnarray*}
   & & \cl (\sigma ) - 2(A_\sigma + A_{\sigma^{-1}} + B_\sigma ) - \frac{1}{2}\cdot\#\{i;\, \sigma (i) \not=i\} \\ &=& \cl (\sigma ) - 2(A_\sigma + A_{\sigma^{-1}} + B_\sigma ) - \#\{i;\, i<\sigma (i)\} \\
   &=& C_\sigma + C_{\sigma^{-1}} + \#\{ i; \, \sigma^2 (i) < \sigma (i) <i\}
  \end{eqnarray*}
  Since $\sigma$ is an involution the last summand equals zero and $C_{\sigma^{-1}} = C_\sigma$. Hence the right hand side is $2 \cdot C_\sigma$, which proves the assertion.
 \end{proof}

 \begin{lemma} \label{sgc_odd}
  If $\sigma \in S_g$ is the disjoint product of some transpositions and a cycle of order 3, we get
  \[
   \cl (\sigma) - 2(A_\sigma + A_{\sigma^{-1}} + B_\sigma) = \frac{1}{2}\cdot\#\{i;\, \sigma (i) \not= i\} + \frac{1}{2}
  \]
  if and only if $C_\sigma = C_{\sigma^{-1}} = 0$.
 \end{lemma}
 \begin{proof}
  One can prove this lemma the same way as Lemma \ref{sgc_even}, with only a slight difference: Depending on whether the cycle of order $3$ is increasing or decreasing, one of the terms $\#\{i;\, i<\sigma (i)<\sigma^2 (i)\}, \#\{i;\, \sigma^2 (i) < \sigma (i) < i\}$ equals zero and the other equals one. Thus if we add the two equations
  \begin{eqnarray*}
   \cl (\sigma ) - 2(A_\sigma + A_{\sigma^{-1}} + B_\sigma ) - \frac{1}{2}\cdot\#\{i;\, i < \sigma (i)\} &=& C_\sigma + C_{\sigma^{-1}} + \#\{ i; \, \sigma^2 (i) < \sigma (i) <i\} \\
   \cl (\sigma^{-1} ) - 2(A_\sigma + A_{\sigma^{-1}} + B_\sigma ) - \frac{1}{2}\cdot\#\{i;\, \sigma (i) < i\} &=& C_\sigma + C_{\sigma^{-1}} + \#\{ i; \, i < \sigma (i) < \sigma^2 (i)\}
  \end{eqnarray*}
  we get
  \[
   2 \cl (\sigma) - 4(A_\sigma + A_{\sigma^{-1}} + B_\sigma ) - \#\{i;\, \sigma (i) \not=i\} =  2C_\sigma + 2C_{\sigma^{-1}} + 1
  \]
  and hence the desired result.
 \end{proof}

 Note that since the left hand side of (\ref{sgc}) is an integer and the fraction part of the right hand side is $\frac{1}{2}$, they always differ by at least $\frac{1}{2}$.
 
 \begin{remark}
  One can show that the converse of Lemma \ref{sgc_even} and \ref{sgc_odd} is also true. If the left hand side and right hand side of (\ref{sgc}) are equal resp.\ differ by $\frac{1}{2}$ then $C_\sigma = C_{\sigma^{-1}} = 0$ and $\sigma$ is an involution resp.\ the disjoint product of some transpositions and a cycle of order $3$. Indeed,  for any cycle $\sigma_0$ of order greater than two for which the sequence $(\sign (\sigma_0^{i+1}(1) - \sigma_0^i(1)))_i$ is alternating we get $C_{\sigma_0} \not= 0$ or $C_{\sigma_0^{-1}} \not= 0$ because otherwise the sequence $(|\sigma_0^{i+1}(1) - \sigma_0^i(1)|)_i$ would either increase or decrease, contradicting the fact that it is cyclic. This proves the converse of Lemma \ref{sgc_even}. The argument remains true under the weaker assumption that $(\sign (\sigma_0^{i+1}(1) - \sigma_0^i(1)))_{i=1}^{\ord\sigma_0}$ is alternating after removing some consecutive elements of the same value (Here both terms ``consecutive'' and ``alternating'' refer to the notion that the $\ord\sigma_0$-th element of the sequence has the first element as its successor), which implies the converse of Lemma \ref{sgc_odd}.
  \end{remark}
 
 \begin{definition}
  Let $\sigma \in S_{2g}$ be an involution and $e \in \{1,\ldots, 2g\}$. We say that \emph{$\sigma$ embraces $e$} if there exists an $i$ such that $i < e < \sigma(i)$.
 \end{definition}
 
 Combining our previous calculations from section \ref{sect2} with the combinatorial lemmas we get a classification of the top-dimensional KR-strata.

 \begin{proposition} \label{topdim}
  Let $0\leq d \leq g$, $x = t^{x_0} (u\sigma) \in \Adm^{(d)}$. \smallskip \\
  \emph{(1)} If $g-d$ is even, $x$ is of maximal length if and only if it is possibly maximal, $\sigma$ is a product of $\frac{g-d}{2}$ disjoint transpositions and $C_\sigma = 0$. \smallskip \\
  \emph{(2)} If $g-d$ is odd, $x$ is of maximal length if and only if it is possibly maximal and one of the following conditions on $\sigma$ holds true.
   \begin{enumerate}
    \item[(a)]  $\sigma$ is the product of $\lfloor \frac{g-d}{2} \rfloor$ disjoint transpositions such that $C_\sigma = 0$ and $\sigma$ does not embrace its (unique) fixed point $e$ which is not contained in $\F(x)$ (i.e.\ $u(e) = 1$).
    \item[(b)] $\sigma$ is the product of $\lfloor \frac{g-d}{2} \rfloor - 1$ transpositions and a cycle of order three which are pairwise disjoint such that $C_\sigma = C_{\sigma^{-1}} = 0$.
   \end{enumerate}
 \end{proposition}
 \begin{proof}
  An element $x\in\Adm^{(d)}$ is of maximal length if and only if it is possibly maximal and all inequalities we used to prove Proposition \ref{main}\teil{2} are in fact equalities for $x$, with an exception in the case where $g-d$ is odd. Here we allow that exactly one of these equalities does not hold true, but its left hand side and right hand side differ only by $\frac{1}{2}$ (since we actually estimated $\cl(x)$ against $\frac{g^2+d}{2} = \dim \cA_I^{(d)} + \frac{1}{2}$). Now the mentioned equalities are
  \begin{eqnarray}
   \label{ineq1} \# \{ (i,f) \in \{1, \ldots , g\} \times \F(x);\, i < f < \sigma (i) \} &=& B_\sigma \\
   \label{ineq2} \frac{1}{2}\cdot \#\{i \in \{1, \ldots ,g\};\, \sigma(i) = i \} &=& \frac{d}{2} \\
   \label{ineq3} \cl(\sigma) - 2(A_\sigma + A_{\sigma^{-1}} + B_\sigma) &=& \frac{1}{2} \#\{i;\, \sigma (i) \not= i\}.
  \end{eqnarray}
  Recall that we limited the possible values of an $S_g$-component of an element of maximal length in Lemma \ref{top-dim1}. Thus it suffices to check the assertion in the following cases.
  
  \emph{(1)} Let $g-d$ be even and $\sigma$ the product of $\frac{g-d}{2}$ transpositions. Since the fixed points of $\sigma$ are exactly the elements of $\F(x)$ the first two equations are fulfilled. By Lemma \ref{sgc_even} the third equality is equivalent to $C_\sigma = 0$.
  
  \emph{(2)} Let $g-d$ be odd. We have to consider two cases:
  
  (a) \emph{$\sigma$ is the product of $\lfloor \frac{g-d}{2} \rfloor$ disjoint transpositions.} In this case, we get
  \[
   \frac{1}{2}\cdot \#\{i \in \{1, \ldots ,g\};\, \sigma(i) = i \} = \frac{d}{2} + \frac{1}{2}
  \]
  Hence (\ref{ineq1}) and (\ref{ineq3}) must hold true. Subtracting the left hand side from the right hand side, we see that (\ref{ineq1}) is equivalent to
  \[
   \{i;\, i<e<\sigma (i)\} = 0
  \]
  i.e. $\sigma$ does not embrace $e$. We have already shown that (\ref{ineq3}) is equivalent to $C_\sigma = 0$
  
  (b) \emph{$\sigma$ is the product of $\lfloor \frac{g-d}{2} \rfloor - 1$ transpositions and a cycle of length 3 which are pairwise disjoint:} Since the left hand side of (\ref{ineq3}) is an integer but the right hand side has fraction part $\frac{1}{2}$, they differ by at least $\frac{1}{2}$. Thus $x$ is of maximal length if and only if $C_\sigma = C_{\sigma^{-1}} = 0$ (use Lemma \ref{sgc_odd}) and and $x$ satisfies the equalities (\ref{ineq1}) and (\ref{ineq3}). But those equalities follow from the fact that the fixed points of $\sigma$ are exactly the elements of $\F(x)$.
 \end{proof}
 
 We denote by $M^{(d)} \subset \Adm^{(d)}$ the subset of elements of maximal length.
 
 As an application of this result we can deduce an explicit formula for the number of top-dimensional irreducible components of $\cA_I^{(d)}$ for $d > 0$. The remaining part of this section will not be needed in the sequel.

 We recall the following result of G\"ortz and Yu.
 
 \begin{proposition}[\cite{GY2}, Thm. 7.4]
  If  $\cA_x$ is not contained in the supersingular locus, then the stratum $\cA_x$ is irreducible.
 \end{proposition}
 
 Now let $d>0$. Then the proposition above says that $\cA_x$ is connected for each $x \in \Adm^{(d)}(\mu)$. Thus,
 \[
  \# M^{(d)} = \# \{ \textnormal{ top-dimensional irreducible components of } \cA_I^{(d)} \}.
 \]
 
 \begin{proposition}
  Let $0 \leq d \leq g$. We denote by $C_n := \frac{1}{n+1} \cdot { 2n \choose n }$ the n-th Catalan number. \smallskip \\
  \emph{(1)} If $g-d$ is even,
   \[
    \# M^{(d)} = 2^d \cdot {g \choose d} \cdot C_{\frac{g-d}{2}}.
   \]
  \emph{(2)} If $g-d$ is odd,
   \[
    \# M^{(d)} = 2^d \cdot {g \choose d} \cdot \frac{g-d+1}{2} \cdot C_{\frac{g-d+1}{2}}.
   \]
 \end{proposition}
 \begin{proof}
  The claim is certainly true if $g=d$ as $M^{(g)} = \Adm^{(g)} = \{t^{w(\mu)}; w \in W\}$, so we may assume that $g-d >0$. In order to calculate $\# M^{(d)}$, we have to count the number of $\sigma \in S_g$ of the form described in Proposition \ref{topdim}. As the factor ${g \choose d}$ counts the number of choices of the set of fixed points and $2^d$ is the number of choices for the $\FF_2$-component, we are reduced to the case $d=0$.

  We denote for any positive integer $m$
  \[
   K_m := \{ \sigma \in S_{2m};\, \sigma \textnormal{ is a fixed point free involution}, C_\sigma = 0\}.
  \]
  Mapping an element $\sigma \in K_m$ to the the string with length $2m$ with $i$-th letter $X$ if $\sigma(i) > i$ resp.\ $Y$ if $\sigma(i) < i$ induces a bijection between $K_m$ and the Dyck-words of length $2m$. Thus $\# K_m = C_m$, proving (1).
  
  Now let $g$ be odd. We denote by $K_m^1$ the set of all $\sigma \in S_{2m+1}$ such that $\sigma$ is the product of $m$ disjoint transpositions, does not embrace its fixpoint and satisfies $C_\sigma = 0$. Also, we denote by $K_m^3$ the set of all $\sigma \in S_{2m+1}$ such that $\sigma$ is the product of $m - 1$ transpositions and a cycle of order three which are pairwise disjoint such that $C_\sigma = C_{\sigma^{-1}} = 0$. We write $n := \lfloor \frac{g}{2} \rfloor$. Then by Proposition \ref{topdim},
  \[
   \# M^{(0)} = \# K_n^1 + \# K_n^3.
  \]
  Now we have a bijection
  \begin{eqnarray*}
   K_n^1 & \rightarrow & K_{n+1} \\
   \sigma &\mapsto& \sigma',
  \end{eqnarray*}
  where, denoting the fixed point of $\sigma$ by $e$, the permutation $\sigma'$ is given by
  \[
   \sigma'(i) = \left\{ \begin{array}{ll} \sigma(i) & \textnormal{if } i \not= e,g+1 \\
                                          g+1 & \textnormal{if } i = e \\
                                          e & \textnormal{if } i =g+1
                        \end{array}\right. .
  \]
  In particular, $\# K_n^1 = C_{n+1}$. Now we have a bijection
  \begin{eqnarray*}
   K_n^3 & \rightarrow & \bigcup_{i=0}^{n-1} \{-1, 1\} \times K_i \times K_{n-1-i}^1 \times \{1,\ldots, 2i+1\} \\
   \sigma &\mapsto& (s_\sigma, \sigma_1, \sigma_2, c_1),
  \end{eqnarray*}
  where $c_1 < c_2 < c_3$ are the elements of the cycle of order $3$ of $\sigma$ and
  \begin{eqnarray*}
   s_\sigma &=& \left\{ \begin{array}{ll} -1 & \textnormal{if } \sigma(c_1) = c_3 \\ 1 & \textnormal{if } \sigma(c_1) = c_2 \end{array} \right. \\
   \sigma_1 &=& \sigma_{| [1,c_1) \cup (c_3, g]} \\
   \sigma_2 &=& \sigma_{| (c_1,c_3)}.
  \end{eqnarray*}
  As $c_1$ tells us where to reinsert $\sigma_2$ into $\sigma_1$, this map is indeed bijetive. Thus
  \begin{eqnarray*}
   \# K_3^n &=& \sum_{i=0}^{n-1} 2 \cdot (2i+1) \cdot C_i\cdot C_{n-i} \\
   &=& \sum_{i=0}^{n} (2n+2)\cdot C_i\cdot C_{n-i} - (4n+2)\cdot C_n \cdot C_1 \\
   &=& (2n+2) \cdot C_{n+1} - (4n+2) \cdot C_n \\
   &=& (2n+2) \cdot C_{n+1} - (n+2) \cdot C_{n+1} \\
   &=& n\cdot C_{n+1}
  \end{eqnarray*} 
  This finishes the proof of part (2).
  
 \end{proof}
 
 \begin{corollary}
  Let $d > 0$. Then the number of top-dimensional irreducible components of $\cA_I^{(d)}$ equals
  \[
    2^d \cdot {g \choose d} \cdot C_{\frac{g-d}{2}}
  \] 
  if $g-d$ is even and
  \[
    2^d \cdot {g \choose d} \cdot \frac{g-d+1}{2} \cdot C_{\frac{g-d+1}{2} }
  \]
  if $g-d$ is odd.
 \end{corollary}
 
 \begin{remark}
  One can also give a decription of the number of top-dimensional irreducible components of $\cA_I^{(0)}$. If $g$ is odd then according to Proposition 8.9 in \cite{GY2} there is no supersingular KR stratum among the top-dimensional ones, so that the above formula still holds for $d=0$. If $g$ is even, this proposition tells us that there is exactly one supersingular top-dimensional KR stratum. A formula for its number of connected components is given in \cite{GY1}, Corollary 6.6.
 \end{remark}

 \subsection{The relative position of $p$-rank strata} \label{sectrelpos}

 The aim of this section is to describe the closure of a $p$-rank stratum as a union of KR-strata. The main part of it will be to prove the following theorem:

 \begin{theorem} \label{prank}
  Let $0 \leq d \leq g$. \smallskip \\
  \emph{(1)} If $g-d$ is even,
  \[
   \overline{\cA_I^{(d)}} = \bigcup_{d'\leq d} \cA_I^{(d')}.
  \]
 \emph{(2)} If $g-d \not= 1$ is odd,
  \[
   \overline{\cA_I^{(d)}} = \bigcup_{d'\leq d} \cA_I^{(d')} \setminus \bigcup_{x \in M^{(d'')} \atop d''<d,\, 2 | g-d''} \cA_x.
  \]
 \emph{(3)} $\overline{\cA_I^{(g-1)}} = \bigcup\limits_{x=t^{x_0}(u\sigma) \in \Adm \atop u \not= 0} \cA_x$
  
 \end{theorem}
 Recall that we have shown that the closure of a $p$-rank stratum is again a union of KR strata at the beginning of section \ref{sectgoingup}. 
 
 We begin with the stepwise proof of the first two parts of the theorem. Note that we already proved in Lemma \ref{prank1} that the closure of a $p$-rank stratum and strata corresponding to a higher $p$-rank are disjoint. Now we determine the intersection of $\overline{\cA_I^{(d)}}$ and $\cA_I^{(d-1)}$, which is the main ingredient of the proof of part (1) and (2).

  \begin{proposition} \label{prank2} Let $1 \leq d \leq g$. \smallskip \\
   \emph{(1)} If $g-d$ is even then $\cA_I^{(d-1)} \subset \overline{\cA_I^{(d)}}$. \smallskip \\
   \emph{(2)} If $g-d$ is odd then
   \[
    \overline{\cA_I^{(d)}} \cap \cA_I^{(d-1)} = \cA_I^{(d-1)} \setminus \bigcup_{x \in M^{(d-1)}} \cA_x.
   \]
  \end{proposition}
  
 \begin{proof}
  \emph{(1)} Let $g-d$ be even. We show that every KR-stratum with $p$-rank $d-1$ is contained in the closure of a KR-stratum of $p$-rank $d$.
   
  Let $x = t^{x_0} (u\sigma) \in \Adm^{(d-1)}$. Since $g-d+1$ is odd, $\sigma$ cannot be the product of $\frac{g-d+1}{2}$ disjoint transpositions. So by Corollary \ref{cor going up} there exists an $y \in \Adm^{(d)}$ such that $x \leq y$ and thus $\cA_x \subset \overline{\cA_y}$. \smallskip

   \emph{(2)} Let $g-d$ be odd. Then by Theorem \ref{dim}, we have \[\dim \cA_I^{(d-1)} = \dim \cA_I^{(d)} > \dim \left(\overline{\cA_I^{(d)}} \setminus \cA_I^{(d)}\right).\] Thus no top-dimensional KR stratum of $\cA^{(d-1)}$ can be contained in the closure of $\cA^{(d)}$. Since $\overline{\cA_I^{(d)}}$ is a union of KR strata, we conclude that it is disjoint to any top-dimensional stratum in $\cA_I^{(d-1)}$.

 So we are left to show that any non-top-dimensional KR stratum $\cA_x \subset \cA_I^{(d-1)}$ is contained in the closure of $\cA_I^{(d)}$. We prove this by constructing for every $x \in\Adm^{(d-1)}\setminus M^{(d-1)}$ a $y \in\Adm^{(d)}$ which dominates $x$ w.r.t.\ the Bruhat order. Now let $x = t^{x_0}(u\sigma) \in \Adm^{(d-1)}$ be not of maximal length. We distinguish between three different cases:
   
  (a) \emph{$\sigma$ is not an involution with $d-1$ fixed points.} Then $x$ is dominated by an element of $\Adm^{(d)}$ by Corollary \ref{cor going up}.
   
  (b) \emph{$\sigma$ is an involution with $d-1$ fixed points, $C_\sigma > 0$.} Recall that every element of $\Adm^{(d-1)}$ is dominated by a possibly maximal element of $\Adm^{(d-1)}$ which has the same $S_g$-component. So we can assume without loss of generality that $x$ is possibly maximal. Since $C_\sigma > 0$, we find $i,j$ such that $i<j<\sigma(i)<\sigma(j)$. We write $a_i = \sigma (i)$ and $a_j = \sigma(j)$. Let $y' = t^{y'_0} w' = t^{y'_0} (u'\sigma')$ be defined as $s_{j, a_i} \cdot x$. We have to check the $\mu$-admissibility criterion of Lemma \ref{mu-adm} at the coordinates $i,j,a_i$ and $a_j$.
   \begin{twoeqnarray}    
    y'_0(i) = x_0(i) &=& 0 & w'^{-1} (i) &=& a_i > i \\
    y'_0(j) = x_0(a_i) &=& 1 & w'^{-1} (j) &=& i < j \\
    y'_0 (a_i) = x_0 (j) &=& 0 & w'^{-1} (a_i) &=& a_j > a_i \\
    y'_0 (a_j) = x_0 (a_j) &=& 1  &\quad w'^{-1} (a_j) &=& j < a_j        
   \end{twoeqnarray}
   So $y'$ is $\mu$-admissible. Obviously $w'$ has $d-1$ fixed points and $\sigma'$ is not an involution. Hence $y'$ is dominated by some $y \in\Adm^{(d)}$ by Corollary \ref{cor going up}. We have $x\leq y'$ by Lemma \ref{wall} since
   \[
    \langle \beta^1_{j,a_i} , x_0 \rangle = x_0 (j) - x_0 (a_i)= -1,
   \]
   so $y$ also dominates $x$.
    
  (c) \emph{$\sigma$ is an involution with $d-1$ fixed points, $C_\sigma = 0$.} Since $x$ is not of maximal length, it is not possibly maximal. However, using Lemma \ref{GY82} we may (and will) assume there is only a single coordinate $1 \leq c_1 \leq g$ such that $u(c_1) = 1$ and $\sigma(c_1) \not= c_1$. Let $c_2 = \sigma (c_1)$. Then $\sigma$ interchanges $c_1$ and $c_2$ and $u(c_2)=0$. Now $y = s_{c_1, c_2}\cdot x$ is obviously contained in $\Adm^{(d)}$ and thus also $x \leq y$ by Corollary \ref{lastminute}.
  \end{proof}
 
 In order to deduce a comparison between arbitrary $p$-rank strata from this proposition we need the following lemma.

 \begin{lemma} \label{prank3}
  If $d \not= g-1$, the stratum $\cA_I^{(d-2)}$ is contained in $\overline{\cA_I^{(d)}}$.
 \end{lemma}
 \begin{proof}
  Let $x = t^{x_0} (u\sigma) \in \Adm^{(d-2)}$. We assume without loss of generality that $x$ is possibly maximal. If there exists a $Z\in\Z(x)$ of order 2 we done by Lemma \ref{going up}\teil{2}.
  
  Assume none of the cycles has order two. Then we can apply Lemma \ref{going up}\teil{1} or \mbox{\ref{going up}\teil{3}} to it. We get an element $y' \in \Adm^{(d-1)}$ which dominates $x$. If we applied Lemma \mbox{\ref{going up}\teil{1}}, then $\Z(y')$ does not contain a cycle of order $2$, if we applied Lemma \ref{going up}\teil{3} there is at most one. So unless $g-d = 1$ we can apply Corollary \ref{cor going up} to $y'$ to obtain a $y \in \Adm^{(d)}$ such that $x\leq y'\leq y$.
 \end{proof}

 \begin{proof}[Proof of Theorem \ref{prank}\teil{1-2}]
  Successive application of Lemma \ref{prank3} yields that
  \begin{equation} \label{succprank2}
   \overline{\cA_I^{(d-2d')}} \subset \overline{\cA_I^{(d)}}
  \end{equation}
  for any $d' \leq \lfloor \frac{d}{2} \rfloor$. Now we combine this assertion with Proposition \ref{prank2}. \smallskip 
  
  \emph{(1)} Let $g-d$ be even. Then Proposition \ref{prank2} says that
  $
   \cA_I^{(d-2d'-1)} \subset \overline{\cA_I^{(d-2d')}}
  $
  for any $d' \leq \lfloor \frac{d-1}{2} \rfloor$. But then (\ref{succprank2}) implies that
  \[
   \cA_I^{(d')} \subset \overline{\cA_I^{(d)}}
  \]
  for any $d' \leq d$. Thus we have $\bigcup_{d'<d} \cA_I^{(d')} \subset \overline{\cA_I^{(d)}}$. Since the converse inclusion was proven in Lemma \ref{prank1} this gives the desired result. \smallskip
  
  \emph{(2)} Let $g-d$ be odd. Then Proposition \ref{prank2} states that
  \[
   \cA_I^{(d-2d'-1)} \setminus \bigcup_{x \in M^{(d-2d'-1)}} \cA_x \subset \overline{\cA_I^{(d-2d')}}
  \]
  for every $d' \leq \lfloor \frac{d-1}{2} \rfloor$. In combination with (\ref{succprank2}) this implies that
  \[
   \overline{\cA_I^{(d)}} \supset \bigcup_{d'\leq d} \cA_I^{(d')} \setminus \bigcup_{x \in M^{(d'')} \atop d''<d,\, 2 | g-d''} \cA_x.
  \]
  Since we know by Lemma \ref{prank1} that $\overline{\cA_I^{(d)}} \subset \bigcup_{d' \leq d} \cA_I^{(d')}$, the claim is reduced to the following lemma:

  \begin{lemma} \label{prank4}
  Let $g-d$ even and $u < g-d$ be an odd integer. Then none of the top-dimensional KR-strata $\cA_x \subset \cA_I^{(d)}$ are contained in the closure of $\cA_I^{(d+u)}$.
  \end{lemma}
  \begin{proof}
  Assume the contrary: Then there exist $x \in \Adm^{(d)}$ of maximal length and $y \in \Adm^{(d+u)}$ with $x \leq y$, i.e.\ we can find reflections $s^{(1)}, \ldots , s^{(k)}$ such that $y = s^{(k)} \cdot \ldots \cdot s^{(1)} \cdot x$ and $\cl (s^{(l+1)} \cdot \ldots \cdot s^{(1)} \cdot x) = \cl (s^{(l)} \cdot \ldots \cdot s^{(1)} \cdot x) +1$ for every $1\leq l < k$. But with each of these reflections we gain at most two additional fixed points. Thus we get $k \geq \frac{u+1}{2}$ and
  \[
   \cl(y) = \cl (x) + k \geq \frac{g^2+d}{2} + \frac{u+1}{2} = \left\lfloor \frac{g^2 + d+ u}{2} \right\rfloor  + 1,
  \]
  contradicting Theorem \ref{dim}.
 \end{proof}

 \noindent This finishes the proof of Theorem \ref{prank}\teil{1-2}.
 \end{proof}
 
 In order to prove the third part, we need the following lemma. 
 
 \begin{lemma} \label{reduced decomposition}
  Let $w\in W$ and $t^{w(\mu)} = s^{(\frac{g(g+1)}{2})}\cdot \cdots \cdot s^{(2)} \cdot s^{(1)} \cdot \tau$ a reduced decomposition. Then $\# \{j;\, s^{(j)} = s_0 \textnormal{ or } s^{(j)}=s_g\} = g$. 
 \end{lemma}
 
 \begin{proof}
  We denote
  \begin{eqnarray*}
   a &:=& \#\{j;\, s^{(j)} = s_0 \textnormal{ or } s^{(j)} = s_g\} \\
   b &:=& \#\{j;\, s^{(j)} = s_{i,i+1} \textnormal{ for some } i\}
  \end{eqnarray*}
 Then $a+b=\cl(t^{w(\mu)}) = \frac{g(g+1)}{2}$.
 
 Recall the description of the $s_i$ in subsection \ref{GSp1}. (Left) multiplication with $s_0$ or $s_g$ changes exactly one coordinate of the $\FF_2^g$-component, multiplication with $s_{i,i+1}$ only permutes its coordinates; in particular it does not change the number of coordinates which are zero resp.\ one. Since the $\FF_2^g$-component of $t^{w(\mu)}$ resp.\ $\tau$ is $(0^{(g)})$ resp.\ $(1^{(g)})$ we get $a \geq g$.
 
 On the other hand, multiplication with $s_0$ or $s_g$ does not change the $S_g$-component and multiplication with $s_{i,i+1}$ induces a multiplication with an adjacent transposition on the $S_g$-component. Since the $S_g$-component of $t^{w(\mu)}$ resp.\ $\tau$ is $\id$ resp.\ $(1\enspace g)(2\enspace g-1)\cdots$ we get $b \geq \cl( (1\enspace g)(2\enspace g-1)\cdots )= \frac{g(g-1)}{2}$, proving our claim. 
 \end{proof}

 \begin{proof}[Proof of Theorem \ref{prank}\teil{3}]
  Recall that for $x,y \in \widetilde{W}$, we have $x \leq y$ if and only if $x$ can be realized as a subsequence of a reduced decomposition of $y$. We know that an element $y = t^{y_0}\cdot (u'\sigma')\in\Adm$ is contained in $\Adm^{(g-1)}$ if and only if $u' = (0^{(i-1)}\, 1\, 0^{(2g-i)})$ for some $i$ and $\sigma' = \id$. Using Lemma \ref{reduced decomposition} and the observations made in its proof, we see that this is equivalent to the claim that $y$ can be realised by removing one $s_0$ or $s_g$ from a reduced decomposition of $t^{w(\mu)}$ for some $w\in W$. Thus $x = t^{x_0}(u\sigma)\in\Adm$ is dominated by an element $y\in\Adm^{(g-1)}$ if and only if $x$ can be realized as a subsequence of a reduced decomposition of some $t^{w(\mu)}$ with at most $g-1$ reflections equal to $s_0$ or $s_g$. In particular we get that $u \not= 0$.
  
 Now let $x=t^{x_0}\cdot(u\sigma)$ with $u \not= 0$. By repeatedly applying Lemma \ref{GY82} and Lemma \ref{going up}\teil{1}, we may assume that $u(i_0)=1$ for exactly one integer $i_0$. Then $i_0 \not\in \F(x)$, so it occurs in a unique cycle $Z = (i_0 \, \cdots\, i_n)$ of $x$. We proceed by induction on $\ord Z$.
 
  If $\ord Z =1$, then $x$ is possibly maximal. By subsequently applying Lemma \mbox{\ref{going up}\teil{2-3}} we get a possibly maximal element $y$ that dominates $x$ such that $\Z(y) = \{(i_0)\}$ and thus $y \in \Adm^{(g-1)}$.
  
  If $\ord Z \geq 2$, let $\widetilde{x} = t^{\widetilde{x}_0}\cdot\widetilde{w} = t^{\widetilde{x}_0}\cdot (\widetilde{u} \widetilde{\sigma}):=s_{i_0, i_1} \cdot x$. Then
  \begin{eqnarray*} 
   \widetilde{u}(j) &=& \left\{\begin{array}{ll} 0 & j \not= i_1 \\ 1 & j=i_1 \end{array}\right. \\
   \F(\widetilde{x}) &=& \F(x) \cup \{i_0\} \\
   \Z(\widetilde{x}) &=& \Z(x) \setminus \{Z\} \cup \{Z'\}
  \end{eqnarray*}
  where $Z' = (i_1\,\cdots\, i_n)$, in particular $\ord Z' = \ord Z -1$. To prove that $\widetilde{x}$ is $\mu$-admissible, it suffices to check the criterion of Lemma \ref{mu-adm} at the coordinate $i_1$. Indeed, we have $\widetilde{x_0}(i_1) = x_0(i_0) = 1$ and $\widetilde{w}^{-1}(i_1) > g \geq i_1$ because of $\widetilde{u}(i_1) = 1$. Now $x\leq \widetilde{x}$ by Corollary \ref{lastminute}. Using the induction hypothesis we find that there exists a $y\in\Adm^{(g-1)}$ such that $x\leq\widetilde{x}\leq y$.
 \end{proof}

\end{document}